\documentclass[12pt,a4paper]{amsart}

\usepackage{latexsym}
\usepackage{amsmath}
\usepackage{amssymb}
\usepackage{amsthm}
\usepackage{amsxtra}
\usepackage{amsfonts}
\usepackage{mathrsfs}
\usepackage{amsbsy}
\usepackage{graphicx}
\usepackage{hyperref}

\usepackage{stackengine}

\newtheorem{teo}{Theorem}[section]
\newtheorem{dfn}[teo]{Definition}

\newtheorem{cor}[teo]{Corollary}
\newtheorem{prop}[teo]{Proposition}
\newtheorem{remark}[teo]{Remark}
\numberwithin{equation}{section}

\renewcommand{\Re}{\operatorname{Re}\,}

\newcommand{\HH}{\mathcal H}
\newcommand{\E}{\mathcal E}

\newcommand{\DD}{\mathcal D}

\newcommand{\X}{\mathcal X}
\newcommand{\Z}{\mathcal Z}

\newcommand{\n}{\noindent}

\newcommand{\ve}{\varepsilon}
\newcommand{\erre}{\mathbb{R}}

\newcommand{\al}{\alpha}
\newcommand{\ga}{\gamma}
\newcommand{\Ga}{\Gamma}

\newcommand{\la}{\lambda}

\newcommand{\f}{\frac}
\newcommand{\ba}{\begin{eqnarray}} \newcommand{\ea}{\end{eqnarray}}
\newcommand{\be}{\begin{equation}} \newcommand{\ee}{\end{equation}}
\newcommand{\bdm}{\begin{displaymath}} \newcommand{\edm}{\end{displaymath}} 
\newcommand{\brr}{\begin{array}}\newcommand{\err}{\end{array}}

\newcommand{\lf}{\left}
\newcommand{\ri}{\right}

\newcommand{\bml}{\begin{gather}} 
\newcommand{\eml}{\end{gather}}

\DeclareMathOperator{\Lip}{Lip}

\newcommand{\ran}{\rangle}
\newcommand{\lan}{\langle}

\newcommand{\beq}{\begin{equation}}
\newcommand{\eeq}{\end{equation}}
\newcommand{\RE}{\mathbb{R}}

\def\CO{{\mathbb C}}

\renewcommand{\leq}{\leqslant}
\renewcommand{\geq}{\geqslant}
\renewcommand{\epsilon}{\varepsilon}

\newcommand{\fila}{\phi^{\la}}

\renewcommand{\d}{{\bf d}}
\newcommand{\x}{{\bf x}}
\newcommand{\y}{{\bf y}}
\renewcommand{\k}{{\bf k}}
\newcommand{\0}{{\bf 0}}

\title[]{Well posedness of the nonlinear Schr\"odinger equation\\ with isolated singularities}

\author[]{Claudio Cacciapuoti}
\address{Dipartimento di Scienza e Alta Tecnologia, Universit\`a dell'Insubria, Via Valleggio 11, 22100 Como, Italy}
 \email{claudio.cacciapuoti@uninsubria.it}

\author[]{Domenico Finco}
\address{Facolt\`a di Ingegneria, Universit\`a Telematica
Internazionale Uninettuno,  Corso Vittorio Emanuele II 39, 00186 Roma, Italy}
\email{d.finco@uninettunouniversity.net}

\author[]{Diego Noja}
\address{Dipartimento di Matematica e Applicazioni, Universit\`a
 di Milano Bicocca,  via Roberto Cozzi 55, 20126 Milano, Italy}
\email{diego.noja@unimib.it}

\begin{document}

\begin{abstract} We study the well posedness of the nonlinear Schr\"odinger (NLS) equation with a point interaction and power nonlinearity in dimension two and three. Behind the autonomous interest of the problem, this is a model of the evolution of so called singular solutions that are well known in the analysis of semilinear elliptic equations. We show that the Cauchy problem for the NLS considered enjoys local existence and uniqueness of strong (operator domain) solutions, and that the solutions depend continuously from initial data. In dimension two well posedness holds for any power nonlinearity and global existence is proved for powers below the cubic. In dimension three local and global well posedness are restricted to low powers. 
\end{abstract}

\maketitle

\begin{footnotesize}
 
 %35J10  	Schr\"odinger operator, Schr\"odinger equation
 %35Q55  NLS equations (nonlinear Schr\"odinger) 
 %81Q35  Quantum mechanics on special spaces: manifolds, fractals, graphs, etc.
% 35R02  	Partial differential equations on graphs and networks (ramified or polygonal spaces)
 %37K50  	Bifurcation problems
% 35D35  	Strong solutions to PDEs
% 35A21  	Singularity in context of PDEs
% 37L50   Noncompact semigroups; dispersive equations; perturbations of Hamiltonian systems
 \end{footnotesize}

%\claudio{Le costanti sono indicate delle volte con $c$ delle volte con $C$}

%\claudio{Per indicare la soluzione abbiamo usato sia $u$ che $\psi$. Dobbiamo scegliere una delle due?}

\section{Introduction}

%The only paper treating NLS with a delta potential in dimension $2$ and $3$ we are aware of  is \cite{DN19}. The functional framework of the latter paper however is that of solutions in Colombeau algebra of distributions, and the meaning of delta potential is seemingly quite different from the one accepted in the standard literature and in particular the present analysis. \\
In the present paper we want to define the dynamics of the nonlinear Schr\"odinger equation (or briefly NLS equation) in the presence of isolated singularities.\\ By isolated singularities we mean the solutions of the equation
\begin{equation} \label{cauchy0}
\left\{ 
\begin{aligned}%{ccc} 
i \partial_t\psi   &= - \Delta \psi  \pm|\psi|^{p-1}\psi \\
 \psi (0)  &=   \psi_0 
 \end{aligned} 
 \right.
\end{equation}
where $\psi\in H^2(\RE^n\setminus\{\0\}).$ The sign in front of the nonlinearity will be not important in the sequel.
%and $F$ is a nonlinearity that we will take for simplicity of power type, $F=\pm|\psi|^{p-1}\psi.$ 
Our main result is the local and global well posedness in dimension $n=2$ or $n=3$ when the nature of admitted singularities is suitably restricted. To make more clear the premises of our analysis we describe the analogous and well known problem in the time independent case.
Isolated singularities of semi linear elliptic equations constitute a subject of study flourished at the end of '70s and still prolific, with important ramifications toward quasilinear elliptic and parabolic equations (see for an incomplete but representative bibliography \cite{BreLio81, GKS,JPY94, Lions80, NS1_86, NS2_86, Veron81, Veron96} and references therein). An example especially relevant in our context is given by the stationary equation associated to \eqref{cauchy0}:
\n
\beq\label{nodelta}
-\Delta u  \pm |u|^{p-1} u - \omega u =0 .\ \ \ \ \ \ \ %\text{in }\ \RE^n\setminus\{0\}
\eeq
\noindent
Its positive solutions defined and regular on $\RE^n\setminus\{\0\} $ and vanishing at infinity are called ground states or singular ground states according to their behavior at $\0$, respectively bounded or diverging. 
Consider for example the equation with the minus sign in \eqref{cauchy0} or \eqref{nodelta}, the so called focusing stationary NLS equation. A typical result is the following. Let $1<p<\frac{n}{n-2}$ ($p>1$ if $n=2$); for any singular ground state $u$ of \eqref{nodelta}  %one of the following alternatives (see \cite{Veron81}).
%\begin{enumerate}
%\item There exists an explicit constant $C(p,n)$ such that
%$$\lim_{x\to 0} u(x){|x|}^{\frac{2}{p-1}}=C(p,n)$$
there exist $q\geq 0$ depending on $u$ such that 
\begin{equation*}%\label{q}
\lim_{\x\to \0} u(\x){|\x|}^{{n-2}}=q \ \ \ \ \ (n\geq 3)\ \ \ \text{or} \ \ \ \lim_{\x\to \0} u(\x)\frac{1}{\log{\frac{1}{|\x|}}}=q \ \ \ \ \ (n=2).
\end{equation*}
Moreover  $u$ solves as a distribution
\begin{equation*}%\label{sidelta-}
-\Delta u - |u|^{p-1}u -\omega u = c_n q \delta_\0
\end{equation*}
\vskip3pt
\noindent
for some positive constant $c_n$ depending on the dimension only and that can be absorbed in the singularity. So the only singularities admitted have the behavior of the fundamental solution of the Laplacian $G^0$. 
For  $\frac{n}{n-2}<p<\frac{n+2}{n-2}$ the singular ground state still exists but with a different power type singularity $u_p$ depending on $p$, and for $p\geq\frac{n+2}{n-2}$ equation \eqref{nodelta} has neither a ground state nor a singular ground state (see \cite{NS1_86, NS2_86}). Results for the defocusing (plus sign in \eqref{cauchy0} or \eqref{nodelta}) again show an alternative between a ground state with the singularity of the fundamental solution of the Laplacian $G^0$, a different singularity $u_p$ and finally no ground states at all. The difference with respect to the focusing case is that the two different singularities can coexist in the range $\frac{n}{n-2}<p<\frac{n+2}{n-2}$. We refer to the already quoted references and the detailed monograph \cite{Veron96} for a complete analysis. In the present work we want to study the time dependent NLS equation in the regime in which the singularities around the origin are of type of the fundamental solution $G^0$. It is by no means an obvious fact that this behavior, if initially present, could be preserved by the NLS flow for some range of nonlinearities. To analyze this problem and to achieve our main result, we choose to work in a Hilbert space setting. It turns out that a convenient way to rewrite the equation in \eqref{cauchy0} is as the abstract NLS equation 
\[
i \partial_t\psi   =  \HH_{\al} \psi  \pm|\psi|^{p-1}\psi
\]
where $\HH_\al$ is a self-adjoint operator in $L^2(\RE^n)$. The operator $\HH_{\al} $ is well defined only for $n\leq3$ and it belongs to the class of point interactions (see \cite{Albeverio} and references therein, and Section \ref{ss:PI} below for  the essential facts). This well known class of operators is the most suitable linear part for the abstract NLS equation because elements of the domain $\psi\in \DD(\HH_{\al})$ behave at the origin exactly as the fundamental solution, and they have $H^2$ Sobolev regularity away from the origin. Namely, every element in $\DD(\HH_{\al})$ has the structure
$$
\psi=\phi^\la+qG^\la \qquad (\lambda > 0 )
$$
where $G^\la$ is the Green function of the Laplacian, $(-\Delta+\lambda)G^\la=\delta_\0$, and $\phi^\la\in H^2(\RE^n)$.\\
Moreover a precise relations between the complex coefficient $q$ and the regular part  $\phi^\la\in H^2$ is needed to have a self-adjoint operator with domain $\DD(\HH_{\al})$. In this boundary condition appears a real parameter $\al$ and for any such $\alpha$ one has a different self-adjoint operator $\HH_\al$. Complete definitions will be given in Section \ref{ss:PI}. These facts suggest to treat the nonlinear term as a perturbation in the linear Schr\"odinger dynamics generated by $\HH_{\al}$ and to search for strong solutions of the NLS equation, i.e. solutions of \eqref{cauchy0} with initial data in $\DD(\HH_{\al}).$ The above considerations lead eventually to the Cauchy problem 
\begin{equation} \label{cauchy1}
\left\{ \begin{aligned}
i \partial_t\psi   &=  {\HH_{\al}} \psi   \pm|\psi|^{p-1}\psi \\
 \psi (0)  &=   \psi_0\in \DD(\HH_{\al}) \end{aligned} \right.
\end{equation}
We are not aware of any result about this evolution problem in dimension $n>1$. On the contrary, some rigorous literature exists in the much simpler one dimensional case, where the operator $\HH_\al$ can be interpreted, at least formally, as a Schr\"odinger operator with a delta potential (see \cite{AN09} and the treatment in the more general context of quantum graphs given in \cite{CFN17}). We also mention the paper \cite{MOS18}, where the Cauchy problem with Hartree nonlinearity is treated.
Preliminary to the analysis of well posedness of \eqref{cauchy1} is the construction of some essential technical tools. Namely one has to extend classical interpolation inequalities to a scale of spaces modeled on $\DD(\HH_{\al})$ (so including singular behavior), and to prove dispersive and Strichartz estimates. These properties, for the most part new, are discussed and proved in Section \ref{s:Prel}. The local well posedness is the content of Theorem \ref{local}, the proof of which fills Section \ref{s:wp}, including local existence, unconditional uniqueness, continuous dependence and blow-up alternative. While the proof of the well posedness is insensitive to the actual value of $\al$ (which by this reason and by notational simplicity will be discarded after Section \ref{ss:PI}), the result provides families of singular solutions parametrized by the real $\alpha$. More detailed analysis of the dynamics can of course depend on it. The proof exploits the framework introduced by Kato (see \cite{K, K2, K3}) to study the standard case of the Laplacian, but with different conclusions. In particular, we stress that the two and three dimensional cases display a rather different behavior with respect to the nonlinearity power. The good news is that, as in the standard case of the Laplacian, local well posedness holds for any power in the two dimensional case. On the contrary, in the three dimensional case the admitted powers  are greatly restricted, being in the range $p\in (1,3/2)$, which means a rather mild nonlinearity. We recall that strong $H^2$ solutions exist in the standard case for any power $p$. It seems not possibile to overcome this limitation in the present framework: powers of the logarithm are tame, inverse powers of $|\x|$ are not. We also notice that a consequence of the well posedness for problem \eqref{cauchy1}  is that its solutions solve in distributional sense  the equation 
\begin{equation*}
i \partial_t\psi   =  -\Delta \psi   \pm|\psi|^{p-1}\psi - q\delta_\0\ . \\
\end{equation*}
%admits solutions for appropriately chosen initial data in $\DD(\HH)$.\\
i.e. a NLS equation with a time dependent delta source (see Remark \ref{distributional}). This fact gives a further and suggestive interpretation of both \eqref{cauchy1} and the kind of evolution of singular solutions.
 In Section \ref{s:gwp} global well posedness of the dynamics is treated. Again, for the two dimensional case nothing changes with respect to the standard case of the Laplacian: also in the presence of singular solutions, global existence is guaranteed for nonlinearities $p<3$ for any initial datum. On the other hand, in the three dimensional case global existence holds for the same powers in which local existence is true, $p<3/2$. \\
We conclude this introduction describing possible developments and perspectives raised by the present results, not reduced to the always possible extensions and technical refinements of the results (among the latter we include for example well posedness in $L^2$ or in the form domain of the operator). The problem of singular solutions has an autonomous mathematical interest; however, as recorded in the references cited above, a not secondary physical motivation for their study originated in models of condensed matter, in particular stationary Fermi-Thomas theory and examples of Yang-Mills theories. Due to the considered models it was natural to limit the analysis to the elliptic stationary case. However, we notice that one could be interested in the dynamics of Bose-Einstein condensates in the presence of defects, conveniently modeled as point interactions. In this case the relevant equation is the time dependent one. In particular vortex solutions could fall within this description. A first step is the analysis of existence of stationary solutions and their stability. The first, and in fact unique, result about existence in which it was made use of point interactions appeared in  {\cite{CaCle93} and \cite{CaCle94}. In those papers a branch of singular ground states bifurcating from the linear eigenstate of $\HH$ was proved for the 3D stationary defocusing case in equation \eqref{nodelta}. A more complete classification of standing waves would be desirable, firstly as regards ground states, then including excited (non sign-definite) states and eventually considering the interactions of several singularities (not treated in this paper). Stability of the solutions with respect of the NLS flow is then a natural question. In this respect, we also mention on the physical and modeling side the recent contributions in \cite{Sakaguchi20, Malomed20} that seem to have raised an interest on ``singular solitons''.  A second relevant question concerns a detailed analysis of the evolution for the critical ($p=3$) and supercritical ($p>3$) nonlinearities in the $n=2$ case. One expects blow-up and it would be interesting to understand if and how the singularity plays a role. Finally, scattering theory for the NLS in the presence of point defects is also a relevant issue. About all these problems there is no previous analysis, and they seem to deserve some interest.

\section{Preliminaries \label{s:Prel}}

In this section we fix notations and we prove some technical results used in the following.\\
We denote by $\x$, $\k$ and so on, points in $\RE^n$, $n=2,3$. Correspondingly, we use the notation $x\equiv  |\x|$, $k \equiv |\k|$. \\ 
We denote by $\hat f$ the Fourier transform of $f$, defined to be unitary in $L^2(\RE^n)$: 
\[
\hat f(\k) := \frac1{(2\pi)^{n/2}} \int_{\RE^n} \d\x\, e^{-i\k\cdot\x} f(\x) \qquad \k \in \RE^n. 
\]
%Given the function $F:\CO\to\CO$, with $F=F_1+iF_2$, by $F'$ we mean the derivative in the real sense of the function $(Re z, Im z)\mapsto F_1(Re z,Im z)+i F_2(Re z, Im z)\ .$ 
We denote by $\|\cdot\|$ the $L^2(\RE^n)$-norm associated with the inner product $\langle\cdot, \cdot\rangle$ and with $\|\cdot\|_p$ the $L^p(\RE^n)$-norm while we use $\| \cdot \|_{H^s}$ for the norm in the Sobolev spaces $H^s( \RE^n)$, $s\in\RE$.\\ As customary, we denote with the same symbol $\langle\cdot, \cdot\rangle$ the duality between Banach spaces or evaluation of distributions.
We use sometimes Dirac notation, that is $|u\rangle \langle v |$ stands for the 1-rank operator $f \leadsto u \langle v,f\rangle$.

For all $\lambda>0$ we denote by $G^\lambda$ the $L^2$ solution of the distributional equation $(-\Delta +\lambda )G^\lambda = \delta_{\0}$, where $\delta_{\0}$ is the Dirac-delta distribution centered in $\x=\0$.  Hence, the integral kernel of the resolvent  of the Laplacian is given by 
\[
G^\lambda( \x- \y ) = (-\Delta +\la )^{-1} (\x-\y) \qquad \la\in \RE^+ .
\]
Explicitly we have: 
\[G^\lambda ( \x ) =\left\{ \begin{aligned}
& \frac{1}{2\pi} K_0 (\sqrt{\la}\,x) \qquad & n=2; \\ 
& \frac{ e^{-\sqrt{\la} \, x}}{4\pi x } & n=3. 
\end{aligned}\right.
\]
Here  $K_0$ is the Macdonald function of order zero. We recall the relation $\frac{1}{2\pi} K_0(\sqrt\lambda \, x) = \frac{i}{4} H_0^{(1)}(i\sqrt\lambda \, x)$, where $H_0^{(1)}(z)$ is the Hankel function of first kind and order zero (also known as zeroth Bessel function of the third kind), see, e.g., \cite{was} Eq. (8) p. 78).

We use $c$ and $C$ to denote generic positive constants whose  dependence on the parameters of the problem is irrelevant, their value may change from line to line. 

\subsection{Point interactions\label{ss:PI}} We denote by $\HH_\alpha$  the  self-adjoint operator in $L^2 (\RE^n)$, $n=2,3$, given by the Laplacian with a delta interaction of ``strength'' $\alpha$ placed in the origin. 

We recall that, see \cite{Albeverio}, both for $n=2$ and $n=3$ the structure of the domain of $\HH_\alpha$ is the same: 
\begin{equation}\label{DD}
D (\HH_\alpha) = \left\{ \psi \in L^2 (\RE^n)|\;\psi  = \phi^\la   +q \,  G^\lambda ,\, \phi^\la \in H^2(\RE^n), \;\; q=\Lambda^\lambda_\alpha \,  \phi^\la(\0)\ri\}
\end{equation}
with 
\[
 \Lambda^{\lambda}_\alpha= \left\{\begin{aligned}
& \f{2\pi}{2\pi \al +\ga +\ln(\sqrt{\la}/2) } \qquad  & n=2, \\
& \f{1}{ \al +\frac{\sqrt{\la}}{4\pi}  } & n=3 ;
\end{aligned} \right. \qquad \alpha \in\RE. 
\]
Here  $ \la $ can be taken in $\RE^+$ (possibly excluded one point which we denote by $-E_\alpha$, where $E_\alpha$ is the negative eigenvalue of $\HH_\alpha$, see below for the details).  For $n=2$,  $\gamma$ is the Euler-Mascheroni constant. The constant $\al$ is real and it parametrizes the family of operators through  $q=\Lambda^\lambda_\alpha \,  \phi^\la(\0)$, which plays the role of a boundary condition at the singularity. For both $n=2$ and $3$ the free dynamics is recovered in the limit $\alpha\to+\infty$. %We recall also that $\alpha$ is related to the s-wave scattering length $a_0$ through the relation:  $a_0 = (-2\pi\alpha)^{-1}$ for $n=2$  and $a_0 = -(4\pi\alpha)^{-1}$ for $n=3$.
\\
% We remark that if $\psi$ belongs to $D(\HH_\alpha)$, then, in general, $\psi$ depends on $\alpha$,  even though we do not make explicit this  dependence. A similar remark holds true for $\phi^\lambda$ (the ``regular part'' of a generic function in $D(\HH_\alpha)$). \\
The action of the operator  is given by 
\begin{equation}\label{action1}
(\HH_\alpha +\la)\psi = (-\Delta +\la )\phi^\la \qquad \forall\psi\in D(\HH_\alpha). 
\end{equation}
\n
The Hamiltonian $\HH_\alpha$ has $[0,\infty)$ as continuous spectrum  furthermore there is no singular continuous spectrum. For $n=2$, $\HH_\alpha$ has  a simple negative eigenvalue $\{ E_\al \}$ for any $ \al \in \RE$. For $n=3$,  if $\al\geq 0$ there is no point spectrum, while for $\al<0$ there is a simple negative eigenvalue $\{ E_\al \}$. Whenever the eigenvalue $E_\alpha$ exists, we denote by $\psi_\alpha$ the corresponding eigenvector. Explicitly one has:
\[
\begin{aligned}
& E_\al= -4 e^{-2(2\pi \al +\ga)} \qquad && \psi_\al (\x) =    \frac{1}{2\pi} K_0 (  2\, e^{-(2\pi \al +\ga)}x) \qquad && n=2; \\ 
& E_\al= -(4 \pi \al)^2   && \psi_\al (\x) =  \frac{ e^{4 \pi \al \, x}}{4\pi x},  \quad \alpha<0&&   n=3.
\end{aligned}
\]
The resolvent of $\HH_\alpha$ is given by the  abstract Kre\u{\i}n resolvent formula 
\begin{equation}\label{resolvent}
(\HH_\alpha+\lambda)^{-1}  = (-\Delta +\lambda)^{-1} +(\Lambda_\alpha^\lambda)^{-1} |G_\lambda\rangle \langle G_\lambda| \qquad \lambda \in \RE^+\backslash\{|E_\alpha|\}  
\end{equation}
($\lambda \in  \RE_+$ if  $n=3$ and $\alpha \geq 0$). 

%Let us also introduce the sesquilinear form $F_\alpha$ on $L^2(\RE^n)$ with domain and action given by
%\begin{align}
%D(F_\alpha)&=\{\psi\in L^2(\RE^n)\ |\ \exists Q_{\psi}\in \CO:\ \phi^\lambda\equiv\psi-Q_{\psi}G^\lambda\in H^1(\RE^n) \}\\
%F_\alpha(\psi_1,\psi_2)&= F^{\lambda}(\psi_1,\psi_2)+\left(\Lambda^{\lambda}_\alpha\right)^{-1}\overline Q_{\psi_1}Q_{\psi_2}
%\end{align}
%where
%\beq
%F^{\lambda}(\psi_1,\psi_2)=\langle\nabla \phi^\lambda_1, \nabla\phi^\lambda_2\rangle +\lambda\langle\phi^\lambda_1, \phi^\lambda_2\rangle -\lambda\langle\psi_1, \psi_2\rangle.
%\eeq
%The corresponding quadratic form is again denoted as $F_\alpha$. As it is well known, it does not depend on $\lambda$, it is symmetric, closed, bounded from below and it is uniquely associated to the operator $\HH_\alpha$, in the sense that $F_{\alpha}(\psi_1,\psi_2)= \langle\psi_1, \HH_\alpha\psi_2\rangle$, $\forall \psi_1 \in D(F_\alpha)$ and  $\forall  \psi_2\in D(\HH_\alpha)$.\\

For  $\la > |E_\al |$ (take $\la >0 $ if $n=3$ and $\al\geq 0$)  we  define 
\[
\DD_\alpha^s: = D\big( ( \HH_\alpha+\la)^s \big), \qquad s \in \RE
\]
where $( \HH_\alpha+\la)^s$ is a self-adjoint operator on $L^2(\RE^n)$ defined through functional calculus. We equip $\DD_\alpha^s$ with the norm  $\|\psi\|_{\DD^s_\alpha} : = \|( \HH_\alpha+\la)^s \psi\|$, equivalent to the graph norm of the operator $( \HH_\alpha+\la)^s $. Consistently we set
\begin{equation}\label{DDalpha}
\DD_\alpha \equiv \DD_{\alpha}^{s=1} = D(\HH_\alpha)
\end{equation}
defined in Eq. \eqref{DD}.  
\begin{remark}
Given a function $\psi\in\DD_\alpha$, we refer to $\phi^\lambda$ (defined in Eq. \eqref{DD}) as its regular part. By Eq. \eqref{action1},  $\|\psi\|_{\DD_\alpha} = \|( \HH_\alpha+\la) \psi\| =  \|(-\Delta +\lambda) \phi^\lambda\|$ which  is equivalent to $\|\phi^\la\|_{H^2}$. Since $\|\cdot\|_{\DD_\alpha}$ is also equivalent to the graph norm of $\HH_\alpha$, then $\|\cdot\|_{\DD_\alpha}$ is equivalent to the $H^2$-norm of the regular part. From now on we shall always use the notation \eqref{DDalpha} for the set $D(\HH_\alpha)$, and work with the norm $\| \cdot \|_{\DD_\alpha}$.
\end{remark}
From now on, to simplify the notation, and since $\alpha$ is regarded as a fixed parameter, we omit $\alpha$ from the notation for objects that may depend on it  and we simply write, for example,  $\HH\equiv \HH_\alpha$, $\Lambda \equiv \Lambda_\alpha$, and $\DD\equiv \DD_\alpha$.

\subsection{Embeddings and interpolation inequalities}

 We recall the  Sobolev embedding (see, e.g., \cite[Th. 2.8.1 $b)$ and $e)$, and Rem. 2 to the theorem]{Triebel}): 
\be \label{r:embedding}\begin{aligned}
& H^s(\RE^n) \hookrightarrow L^q(\RE^n) \qquad & 2\leq q <\infty,\;   s \geq s_c; \\ 
& H^s(\RE^n) \hookrightarrow C_B(\RE^n) & s>n/2;  
\end{aligned}
\ee
where $s_c= n(\frac12 -\frac1q)$ and  $C_B(\RE^n)$ denotes the space of bounded, continuous functions on $\RE^n$.  
%Hence, setting $q = 2p$, we have the embedding 
%\be \label{semplice}
% H^s(\RE^n) \hookrightarrow L^{2p}(\RE^n)  \qquad s\geq \frac{n(p-1)}{2p} .
%\ee
We will need further embedding properties involving the domains of operators $\HH+\lambda$ and the domains of their fractional powers. Recall that, for all $\lambda >0$, 
\[\begin{aligned}
&G^\lambda  \in L^s(\RE^3) \qquad && 1\leq s <3; \\ 
& G^\lambda \in L^s(\RE^2)  && 1\leq s <\infty.
\end{aligned}
\]
Hence, by the definition of the operator domain $\DD$, see Eq. \eqref{DD}, it follows that 
\begin{equation}\label{DDalpha-embedding}
\DD \hookrightarrow L^{q}(\RE^n) \qquad \text{where $2\leq q <3$ if $n=3$, and $2\leq q <\infty $ if $n=2$.}
\end{equation}
A less obvious property is given in the following:
\begin{prop} \label{frazio}
We have $\DD^{s/2} \hookrightarrow H^{s}$ with continuous  embedding when:
\begin{align*}
 & 0<s<1 && \text{if }n=2;  \\
 &0<s<1/2 && \text{if }n=3. 
\end{align*}
\end{prop}
\begin{proof}[{\bf Proof}]
%We use several results and ideas from \cite{GMS}. 
Let us start from an abstract result. Recall the integral identity (see, e.g., \cite[Ch. 10.4]{BS}): 
\[
x^{s/2} = \frac{\sin(\frac{s}{2}\pi)}{\pi} \int_0^{+\infty} dt\, t^{\frac{s}{2}-1} \frac{x}{t+x} \qquad x\geq 0,\;s\in(0,2).
\]
The latter, applied to $x=(y+\lambda)^{-1}$, $(y+\lambda)>0$, an by means of  the change of variables $t \to 1/t$,  gives 
\[
(y+\lambda)^{-s/2} = \frac{\sin(\frac{s}{2}\pi)}{\pi} \int_0^{+\infty} \frac{dt}{t^{s/2}} \; (y+\lambda +t)^{-1} \qquad (y+\lambda)> 0,\;s\in(0,2).
\]
Hence, by functional calculus and by the Kre\u{\i}n resolvent formula (see Eq. \eqref{resolvent}), one infers  (see \cite{GMS}): 
\[
 ( \HH+\la)^{-s/2} = (-\Delta +\la)^{-s/2} +\frac{\sin \frac{s}{2}{\pi}}{\pi} \int_0^{+\infty}  \frac{dt}{t^{s/2}} \;   (\Lambda^{\lambda+t})^{-1} |G_{\la +t} \rangle \langle G_{\la+t} | \qquad \lambda >|E_\alpha|
\]
($\lambda >0$ if  $n=3$ and $\alpha \geq 0$).

Starting from the latter identity, for $n=3$ and $\al \geq 0 $ a stronger result was proven in \cite{GMS}, see Theorem 3.2, namely that $\DD^{s/2} = H^{s}$ 
for $0<s<1/2$ and that the graph
norm associated to $( \HH+\la)^{s/2}$ is equivalent to the  standard Sobolev norm. 

Being  $\psi \in \DD^{s/2}$ iff $\psi = ( \HH+\la)^{-s/2} f$ for $f\in L^2$, any function in $\DD^{s/2}$ can be written as $\psi= \psi_1 + \psi_2$ with $\psi_1 =  (-\Delta +\la)^{-s/2} f$  and
\[
\psi_2 = \frac{\sin \frac{s}{2}{\pi}}{\pi} \int_0^{+\infty}      \frac{dt}{t^{s/2}} \, (\Lambda^{\lambda+t})^{-1} G_{\la +t} \, g(t) \qquad  
g(t) = \langle G_{\la+t},f \rangle.
\]
Since $\psi_1$ manifestly belongs to $D((-\Delta+\lambda)^{s/2})$ and  the graph norm of $(-\Delta+\lambda)^{s/2}$  is equivalent to the $H^{s}$-norm,   we turn our attention to $\psi_2$. Taking the Fourier transform we have
\[
\hat \psi_2(k)  = \frac{\sin \frac{s}{2}{\pi}}{\pi} \int_0^{+\infty}   \frac{dt}{t^{s/2}}   (\Lambda^{\lambda+t})^{-1} \frac{ g(t)}{k^2+\la +t}\ ,  \qquad  
g(t) =\int_{\RE^3} \d\k \frac{1}{k^2+\la + t} \hat{f} (\k).
\]

We want to prove that $\psi_2\in H^{s}$. For $n=3$ and  $\al <0$ notice that proofs of Lemma 5.1 and Proposition 5.2 of \cite{GMS} hold true without modifications if one assumes
$\lambda > |E_\al|$. Lemma 5.1 is actually   unrelated to the value of $\alpha$ and Proposition 5.2 uses only the fact that $\sup_{t>0}  (\Lambda^{\lambda+t})^{-1} <\infty$, which is indeed the case if $\lambda>|E_\alpha|$. This completes the proof of the  three dimensional case.

Let us consider the two dimensional case. First we prove that 
\be \label{tonno}
 \int_0^{+\infty} dt\  |g(t)|^2 \leq c \|f\|^2.
\ee
We start by noticing that 
\[
\int_0^{+\infty} dt  \lf| \int_{\RE^2} \d\k \frac{1}{k^2+\la + t} \hat{f} (\k)\ri|^2 =
\int_0^{+\infty} dt  \lf| \int_{0}^{+\infty}  dk \frac{\sqrt{k}}{k^2+\la + t} \sqrt{k}A\hat{f} (k)\ri|^2
\]
where
\[
A\hat{f} (k)= \int_0^{2\pi} d\theta \hat{f}(k,\theta) .
\]
Notice that $ \sqrt{k}A\hat{f} \in L^2(\RE^+) $ and that $ \| \sqrt{k}A\hat{f} \|_{ L^2(\RE^+)} \leq \sqrt{2\pi} \|f\| $. Then, to complete the proof of Eq.  \eqref{tonno}, it is sufficient to prove that
\[
T_1 (t,k) = \frac{\sqrt{k}}{k^2+\la + t} 
\]
is the integral kernel of a bounded operator $T_1: L^2 (\RE^+) \to L^2(\RE^+)$.  To this aim, let us notice that, by scaling $t \to t/(k^2+\lambda)$ in the integral we have  
\[
\sup_{k>0}   \int_0^{+\infty} dt\,T_1 (t,k) \f{1}{ t^{1/4}}
\leq  \sup_{k>0} \frac{\sqrt{k}}{(k^2 + \lambda)^{\frac14}}   \int_0^{+\infty} dt\, \frac{1}{ t+1}  \f{1}{ t^{1/4}} <\infty  
\]
and, by scaling $k \to k/\sqrt t$ in the integral, 
\[
\sup_{t>0} t^{1/4} \int_0^{+\infty} dk\,T_1 (t,k)   <  \sup_{t>0} t^{1/4} \int_0^{+\infty} dk\, \frac{\sqrt{k}}{k^2 + t} =   \int_0^{+\infty} dk\, \frac{\sqrt{k}}{k^2 + 1}<\infty ,
\]
then  the claim follows from  Schur's test, see, e.g., \cite{Hal}.
%This follows from Schur's test, see (A.7) in \cite{GMS} and take $\ga=1$, $\de=2$ and $\beta=0$.

Next we prove that $\hat \psi_2$ is in $H^{s}$. Precisely, we are going to prove that $(-\Delta+\lambda)^{s/2} \psi_2\in L^2$. To this aim we shall show that  $\| (-\Delta+\lambda)^{s/2} \psi_2 \| \leq c  \| g\|_{L^2 (\RE^+)}^2$ and then use  the inequality  \eqref{tonno}. Since $\hat \psi_2$ is spherically symmetric, this is equivalent to prove that  $\sqrt{k} (k^2+\la)^{s/2} \hat \psi_2 $ belongs to $L^2(\RE^+)$. 
Using the above definitions, we have
\[% \label{tacchino}
\sqrt{k} (k^2+\la)^{s/2} \hat \psi_2(k)  = \sqrt{k} \, \frac{\sin \frac{s}{2}{\pi}}{\pi} \int_0^{+\infty}   \frac{dt}{t^{s/2}}    \f{2\pi}{2\pi \al +\ga +\ln(\sqrt{\la+t}/2) }\frac{ (k^2+\la)^{s/2} }{k^2+\la +t}g(t) .
\]
Since $t>0$, for all $\lambda > |E_\alpha|$ there exists a constant $c$ such that
\[
0< \f{2\pi}{2\pi \al +\ga +\ln(\sqrt{\la+t}/2) } <\f{2\pi}{2\pi \al +\ga +\ln(\sqrt{\la}/2) } \leq c.
\]
 Hence, it is sufficient to prove that $T_2: L^2 (\RE^+) \to L^2(\RE^+)$ defined by the integral kernel  
\[
T_2 (t,k) = \frac{k^{\frac12}(k^2+\lambda)^{s/2}  }{t^{s/2}(k^2+\la + t)}
\]
 is  a bounded operator.
%This follows again from Schur's test, see (A.7) in \cite{GMS} with  $\ga=1$, $\de=2$ and $\beta=-s$.
By scaling $t \to t/(k^2+\lambda)$ in the integral we have, on one hand, 
\[
\sup_{k>0}  \sqrt{k} \int_0^{+\infty} dt\,T_2 (t,k) \f{1}{\sqrt{t}} = \sup_{k>0} \frac{k}{\sqrt{k^2+1}} \int_0^{+\infty}  dt\,\frac{1}{t^{\frac{s}{2}+\frac12}(1+t)}<\infty. 
\]
On the other hand, by scaling $k \to k/\sqrt{\lambda +t}$ it is easy to see  that 
\[\begin{aligned}
& \sup_{t>0} \sqrt{t} \int_0^{+\infty} dk\,T_2 (t,k)  \f{1}{\sqrt{k}} \leq   c\left( \sup_{t>0} t^{\frac12-\frac{s}{2}} \int_0^{+\infty} dk\, \frac{k^{s}}{k^2+\lambda +t} + \lambda^{s/2} \sup_{t>0} t^{\frac12-\frac{s}{2}} \int_0^{+\infty} dk\, \frac{1}{k^2+\lambda +t} \right)  \\ 
= &  c\left( \sup_{t>0} \left(\frac{t}{\lambda+t} \right)^{\frac12-\frac{s}{2}}\int_0^{+\infty} dk\, 
\frac{k^{s}}{k^2+1} 
+ \lambda^{s/2} \sup_{t>0} 
\frac{t^{\frac12-\frac{s}{2}}}{(t+\lambda)^{\frac12} }
\int_0^{+\infty} dk\, \frac{1}{k^2+1} 
\right)
%<\infty, 
\end{aligned}\]
here the constant $c$ depends on $s$. Hence,  the claim follows from  Schur's test.

Then we have
\[
\| (-\Delta+\lambda)^{s/2} \psi_2 \|^2= 2\pi 
\int_0^{+\infty} dk \lf| \sqrt{k} (k^2+\la)^{s/2} \hat \psi_2(k) \ri|^2 \leq c \| g\|_{L^2 (\RE^+)}^2 \leq c\|f\|^2
\]
 and the proof is complete.
 \end{proof}

Thanks to Sobolev embeddings, see Eq. \eqref{r:embedding},  and the results in Proposition \ref{frazio} one has the continuous embeddings
\begin{equation}\label{modifiedembeddings}
\begin{aligned}
& \DD^{s/2}(\RE^2) \hookrightarrow H^s(\RE^2) \hookrightarrow L^q(\RE^2)  \qquad 2\leq q < \infty\ \ \  & s\in [s_c,1) ;\\
& \DD^{s/2}(\RE^3) \hookrightarrow H^s(\RE^3) \hookrightarrow L^q(\RE^3) \qquad  2\leq q < 3\ \ \  & s\in [s_c,1/2) ; 
\end{aligned}
\end{equation}
where $s_c= n(\frac{1}{2}-\frac{1}{q})$, and the corresponding inequalities
\begin{equation}\label{modifiedineq}
\begin{aligned}
\text{if }n=2\qquad& \|\psi\|_{L^q} \leq c\|\psi\|_{H^s} \leq c_2\|\psi\|_{\DD^{s/2}} \quad  \qquad 2\leq q < \infty\ \ \  & s\in [s_c,1); \\
\text{if }n=3\qquad&  \|\psi\|_{L^q} \leq c\|\psi\|_{H^s} \leq c_3\|\psi\|_{\DD^{s/2}}\quad  \qquad  2\leq q < 3\ \ \ \ & s\in [s_c,1/2).
\end{aligned}
\end{equation}

\vskip5pt\noindent
On the other hand $\HH+\la$ is a self-adjoint and positive operator, and the spectral theorem allows to build the scale of Hilbert spaces $\DD^{s/2}$ with the inner product $\langle\psi_1,\psi_2\rangle_{\DD^{s/2}} : = \langle ( \HH+\la)^{s/2} \psi_1, ( \HH+\la)^{s/2} \psi_2 \rangle $. This is a family of real interpolation spaces, in particular this means that (see, e.g., Section 4.3.1 in  \cite{Lunardi18}) 
\begin{equation}\label{interpolationgeneral}
\|\psi\|_{\DD^{(1-\theta) a+\theta b}}\leq c\|\psi\|^{1-\theta}_{\DD^{a}} \|\psi\|^\theta_{\DD^{b}} \qquad  a,b \geq 0,\; \theta\in (0,1) . 
\end{equation}
For $a=0,\ b=1/2,\ \theta=s$ one in particular obtains the inequality
\begin{equation*}%\label{interpolationspecial}
\|\psi\|_{\DD^{s/2}}\leq c\|\psi\|^{1-s}_{L^2} \|\psi\|^s_{\DD^{1/2}}\qquad   s\in (0,1) ;
\end{equation*}
and,  for  $a=0,\ b=1,\ \theta=s$, the inequality  
\begin{equation*}%\label{interpolationspecial2}
\|\psi\|_{\DD^{s}}\leq c\|\psi\|^{1-s}_{L^2} \|\psi\|^s_{\DD}\qquad  s\in (0,1).
\end{equation*}
%\diego{Notice that $a=0,\ b=1,\ \theta=s/2$ yields inequality  used at the beginning of the proof of Proposition \ref{p:hcontinuity}, which should be numbered.}\\
From \eqref{interpolationgeneral} and \eqref{modifiedineq} we finally obtain the Gagliardo-Nirenberg inequalities adapted to the scale of Hilbert spaces $\DD^{s}$:

\begin{equation}\label{adaptedGN}
\begin{aligned}
\text{if }n=2\qquad& \|\psi\|_{L^q} \leq c\|\psi\|^{1-s}_{L^2} \|\psi\|^s_{\DD^{1/2}}\quad  \qquad 2\leq q < \infty\ \ \  & s\in [s_c,1); \\
\text{if }n=3\qquad&\|\psi\|_{L^q} \leq c\|\psi\|^{1-s}_{L^2} \|\psi\|^s_{\DD^{1/2}}\quad  \qquad  2\leq q < 3\ \ \ \ & s\in [s_c,1/2).
\end{aligned}
\end{equation}

\subsection{Evolution operators and space-time estimates} Let us now introduce space-time Banach spaces and several properties of the  evolution operators generated by $\HH$ needed in the sequel. \\ For any exponent $\rho\in[1,+\infty]$ we denote by $\rho'\in[1,+\infty]$ its H\"older conjugate:
\[
\frac1\rho + \frac1{\rho'} =1.
\]
We will denote $L^\rho( [0,T]; L^\sigma (\RE^n))$ by $L^{\rho,\sigma}$ and the corresponding norm by $\|\cdot\|_{\rho,\sigma}$.
The unitary group generated by the operator $\HH$ is denoted by:
\be\label{evol}
U\phi (t) := e^{-it \HH} \phi .
\eeq
The corresponding Duhamel operator is
\be\label{duham}
 \qquad \Gamma u (t) := \int_0^t U(t-s) \, u(s) \, ds . 
\eeq

\begin{dfn}[Admissible pair] 
We say that a pair of (time,space) exponents $(\rho, \sigma)$ is admissible if 
\[\frac2{\rho} + \frac{n}{\sigma} = \frac{n}{2}, \]
and 
\[
\begin{aligned}
& \sigma \in[2,+\infty) \qquad &\text{if $n=2$}; \\
& \sigma \in[2,3) \qquad &\text{if $n=3$}.
\end{aligned}
\]
Correspondingly $\rho\in (2,+\infty]$ if $n=2$ or $\rho \in (4,+\infty] $ if $n=3$.
\end{dfn}
\begin{prop}[Strichartz estimates for $\HH$]  For all the admissible pairs $(\rho,\sigma)$ and $(\mu,\nu)$  there exists a positive constant $C$ such that 
\begin{equation}\label{strichartz1}
\|UP_{ac}(\HH)\phi\|_{\rho,\sigma} \leq C \|\phi\|
\end{equation}
and
\begin{equation}\label{strichartz2}
\|\Gamma P_{ac}(\HH) u \|_{\rho,\sigma} \leq C \|u\|_{\mu',\nu'}
\end{equation}
for all $T>0$. 
\end{prop}
\noindent
These bounds are a direct consequence of the fundamental bound 
\begin{equation}\label{fundamental}
\| U P_{ac} (\HH)\phi\|_\sigma \leq C |t|^{-n(\frac12-\frac1\sigma)} \|\phi\|_{\sigma'} \qquad t\in \RE\backslash\{0\}.
\end{equation}
%see, e.g., the proof of Th. 2.3.3 in \cite{caz}.  \\
\noindent
The proof of the bound \eqref{fundamental} appeared in \cite{DMSY} for $n=3$ and \cite{CMY} for $n=2$ (together with the Strichartz estimates \eqref{strichartz1} and \eqref{strichartz2}, with the time interval $[0,T]$ replaced by $\RE$), see also \cite{DPT} and \cite{IS17}.\\
In the rest of the paper we will find convenient a different parametrization and notation for the admissible pair, obtained changing $\rho$ to $r$ and $\sigma$ to $p+1$.
\begin{dfn}\label{d:rp} For any $p\in[1,+\infty)$ if $n=2$ or $p\in[1,2)$ if $n=3$,  we set \
\begin{equation*}%\label{rp}
r= \f{4(p+1)}{n(p-1)} 
\end{equation*}
so that  $(r,p+1)$ is  a pair of admissible exponents. 
\end{dfn}
%\begin{remark}\label{r>2}
%For all the allowed choices of $p$, one has $r>2$, namely, $r>2$ for $n=2$ and $r>4$ for $n=3$).
%\end{remark}
We summarize in the following proposition the properties of the linear dynamics needed in the proof of the main theorem (see \cite{K,Y}).
\begin{prop} \label{propa}
Let 
\[
\begin{aligned}
&p\in(1,+\infty) \qquad &  \text{if } n=2;  \\
&p\in(1,2) & \text{if }  n=3 .
\end{aligned}
\]
and let $(r,p+1)$ an admissible pair.\\
Then,  the operators $U$ and $\Ga$ are defined and bounded between the following spaces with norms uniformly bounded for $T\leq 1$:
\[
\begin{aligned}
& a)\; U: L^2 \to L^{\infty,2} \qquad \qquad \qquad \qquad \qquad& b)\; & U: L^2\to L^{r,p+1} \\
& c)\; \Ga: L^{1,2}\to L^{\infty,2} & d)\; & \Ga: L^{1,2}\to L^{r, p+1} \nonumber \\ \nonumber
& e)\; \Ga: L^{r',1+1/p} \to L^{\infty,2} & f)\; & \Ga:  L^{r',1+1/p}\to L^{r,p+1}
\end{aligned}
\]
\end{prop} 
\begin{proof}[{\bf Proof}]
Properties $a)$ and $c)$ hold true since $U$ is a unitary operator in $L^2$.
The other properties follow from the spectral theorem together with  Strichartz estimates for the continuous part of the spectrum. Let us give few additional details about the proof.  Let us consider the two dimensional case and write
\be \label{decomp}
U = U_1  + U_2  \qquad U_1  = e^{-it\HH} P_{ac} (\HH) \qquad U_2= e^{-iE_\al t}\frac{|\psi_\al\ran\lan \psi_\alpha|}{\|\psi_\al\|^2}.
\ee
We decompose $\Ga= \Ga_1 +\Ga_2$ accordingly. For the two dimensional case, it was proved in \cite{CMY} that
$U_1$ and $\Ga_1$ satisfy Strichartz  estimates and therefore the remaining properties b) and c)--f) are true since the indexes
$(r,p+1)$ are admissible. Concerning $U_2$ and $\Ga_2$, it is sufficient to notice that $\psi_\al \in L^\sigma$ for $1\leq \sigma <\infty$ and that
we are assuming $T\leq1$; then straightforward calculations using H\"older's inequality give: 
\begin{equation}\label{U2Gamma2}
\text{for } n=2 \qquad 
\begin{aligned}
& \|U_2f\|_{\rho,\sigma} \leq T^{1/\rho} \frac{\|\psi_\alpha\|_\sigma \|\psi_\alpha\|_\gamma}{\|\psi_\alpha\|^2} \|f\|_{\gamma'}  \\
& \|\Gamma_2 u\|_{\rho,\sigma} \leq T^{1/\rho+1/\mu} \frac{\|\psi_\alpha\|_\sigma \|\psi_\alpha\|_\gamma}{\|\psi_\alpha\|^2} \|u\|_{\mu',\gamma'}  
\end{aligned}
\qquad \sigma,\gamma\in[1,+\infty);\; \rho,\mu\in[1,+\infty].
\end{equation}
Which imply $(b)$, $(d)$, $(e)$, and $(f)$ for $U_2$ and $\Gamma_2$. 

In the three dimensional case, for $\al <0$ we argue in the same way using again \eqref{decomp}. Strichartz estimates for $U_1$ and $\Ga_1$
have been proved in \cite{DMSY},  while for $U_2$ and $\Ga_2$ it is important to notice that $\psi_\al \in L^\sigma$ only for $1\leq \sigma <3$. Hence, for $n=3$ bounds of the form \eqref{U2Gamma2} still hold true but with the constraint  $ \sigma,\gamma\in[1,3)$, which causes no problem since to prove $(b)$, $(d)$, $(e)$, and $(f)$ one needs to set $\sigma = 2$ or $\sigma = p+1$, and similarly for $\gamma$.  If $\al \geq 0$ then 
there is no point spectrum and both  $U_2$ and $\Ga_2$ are absent.
\end{proof}
\begin{remark}\label{T_0}
The statement of Proposition \ref{propa} holds true if $T\leq 1$ is changed in $T\leq T_0$, for any positive $T_0$. Obviously, in this case the norms of the operators $U$ and $\Gamma$ would depend on $T_0$. Since in what follows we shall use Proposition \ref{propa} to study the local well-posedness, there is no loss of generality in restricting to the choice $T_0 =1$.  
\end{remark}
\begin{remark}\label{strengthening}
Some of the above properties can be strengthened. In particular,  in Proposition \ref{propa} $c)$ the target space is actually $C([0,T], L^2)$. See Proposition 7.3.4 in \cite{cazhar}).
\end{remark}
\vskip5pt
\noindent
We end this section introducing four Banach spaces needed in the following, and we reformulate dispersive estimates in these spaces. In what follows we assume that $(r,p+1)$ is an admissible couple, according to Definition \ref{d:rp}.  The first couple is given by:
\[
\X = L^{\infty,2} \cap L^{r,p+1} \qquad  \text{and}\qquad\tilde{\X} = L^{1,2} + L^{r',1+1/p} 
\]
with norms defined by 
\[
\| f \|_\X = \max \{\| f\|_{{\infty, 2} }, \| f\|_{{r,p+1} }\} \qquad
\| f \|_{\tilde{\X}} = \inf_{g+h=f} \{\|g\|_{{1,2}} + \| h\|_{{r' , 1+1/p}}\}.
\]
\vskip5pt
%\diego{Penso che dovremmo cambiare leggermente la notazione.\\ Al posto di $\X^{\prime}$ scriverei $\tilde{\X}$ e sottolineerei anche con la formula che $\X=\tilde{\X}^{\prime}$} 

Notice that ${\X}$  is the topological dual of $\tilde{\X}$, i.e. $\X=\tilde{\X}^{\prime}$, and that Proposition \ref{propa} has the following immediate corollary:
\begin{cor}\label{c:UGamma} Under the same assumptions of Proposition \ref{propa} the following holds true:
\be \label{lampada}
U:L^2 \to {\X} \qquad\text{and}\qquad  \Ga :  \tilde{\X}  \to  {\X}, 
\ee
as bounded operators and the operator norms are uniformly bounded for every finite $T$.
\end{cor}
A second couple  of useful spaces is given by:
%\[
%\overline{\Y} = \{ v \in \overline{\X}| \partial_i v  \in \overline{\X} \}
 %\qquad {\Y} = \{ v \in {\X}| \partial_i v  \in {\X} \}
 %\qquad  
%\Y^\prime =  \{ v \in \Y^\prime | \partial_i v  \in \Y^\prime \}
%\]
%\[
%\overline{\Z} = \{ v \in \overline{\X}| \dot v  \in \overline{\X},\, \HH v \in C([0,T];L^2) \}
%\]
\[{\Z} = \{ v \in {\X}|  \partial_t v  \in {\X},\, \HH v \in L^{\infty,2} \}
\qquad \text{and} \qquad 
\tilde{\Z} =  \{ v \in L^{\infty,2}| \partial_t v \in \tilde{\X} \}
\]
with norms given by
\[
\| v \|_\Z = \max \{ \| v\|_{\X },   \| \partial_t v\|_{\X}, \| \HH v \|_{{\infty,2}} \} \qquad
\| v \|_{\tilde{\Z}} =  \max \{ \| v\|_{{\infty,2} },   \| \partial_t v\|_{\tilde{\X}} \}.
\]
In the previous definitions and in the following, the expression $\partial_t v$ is to be interpreted as the distributional derivative of the $Y$-vector valued distribution $v\in \mathscr{D}^{\prime}(I,Y):=\mathcal L(\mathscr D(I), Y)$, where $I$ is an open interval, $\mathscr D:=C_0^{\infty}(I)$ and $Y$ is a relevant Banach space (see for example \cite{cazhar}, sections 1.4.4 and 1.4.5 for details). 
%When convenient for notational reasons, we use with the same meaning the symbol $\partial_t v\ .$\\
%\diego{Ci sta un po' di mixing di simboli per le derivate,  a volte si usa $\dot v $ e a volte si usa $\partial_t v$ . Ho messo l'osservazione precedente per evitare contestazioni.}\\
%\diego{Sarebbe opportuna una notazione meno fuorviante: $\tilde{\Z}$ in luogo di $\Z^{\prime}$. }
 \begin{prop} \label{gatto}
Assume that $p> 1$ if $n=2$ and $1< p<2$ if $n=3$. Then,  the operators $U$ and $\Ga$ are defined and bounded between the following spaces:%\be \label{tre}
%U:\DD \to \overline{\Z} \qquad \Ga :  \tilde{\Z}  \to  \overline{\Z}
%\ee
\[% \label{tre}
U:\DD \to {\Z} \qquad \Ga :  \tilde{\Z}  \to {\Z}
\]
with norms uniformly bounded in $T$:
\begin{align}
&\| U \phi \|_{\Z} \leq c \| \phi \|_{\DD} ;\label{uno} \\
& \| \Ga f\|_\Z \leq c \| f \|_{\tilde{\Z}} .\label{due}
\end{align}
\end{prop}
\begin{proof}[{\bf Proof}]
%It is sufficient to prove \eqref{uno} and \eqref{due} and then \eqref{tre} follows by the same argument used to obtain claim \eqref{cont} from \eqref{lampada}.
By Corollary  \ref{c:UGamma} we have $\| U \phi \|_\X \leq c \|\phi\|$, then by spectral theorem and again Corollary \ref{c:UGamma} we have
\[
\| \HH U\phi \|_{\infty,2} \leq \| \HH U\phi \|_\X=\| \partial_t U\phi \|_\X=  \| U \HH  \phi \|_\X \leq c\|  \HH  \phi \|,
\]
which proves \eqref{uno}.

We note that the obvious  inclusion $L^{\infty,2}\subset L^{1,2}$ implies $\tilde{\Z} \subset \tilde{\X}$, hence,  for $f\in \tilde{\Z} $, we have $\Ga f \in \X$ by Corollary \ref{c:UGamma}. In particular we have
\[
\| \Ga f \|_\X \leq c \| f\|_{\tilde{\X}} \leq c \| f \|_{1,2} \leq c \|f\|_{\infty,2} \leq c\|f\|_{\tilde{\Z}}.
\]
Notice that
\beq\label{identity1}
\partial_t \Ga f= \Ga \partial_t f + U f(0).
\eeq
This identity is justified as an identity in $\X$, whenever $f\in \tilde{\Z}$ by the following argument. Firstly note that by Sobolev embeddings %(see Eq. \eqref{moreembeddings})
there holds true $H^{1}(\RE^n) \hookrightarrow L^{p+1}(\RE^n)$, hence, by duality, $L^{1+1/p}(\RE^n)  \hookrightarrow H^{-1}(\RE^n)$. Which in turn implies  $ L^{r'}([0,T]; L^{1+1/p}) \hookrightarrow L^1([0,T]; H^{-1})$  (because $L^{s}([0,T]) \hookrightarrow L^{1}([0,T])$ for all $s\geq 1$). Moreover, trivially, $L^{1,2}\hookrightarrow L^1([0,T]; H^{-1})$. Hence, $\tilde{X} \hookrightarrow L^1([0,T]; H^{-1})$.   As a consequence,  $\partial_t f \in \tilde{\X} \subset L^1([0,T]; H^{-1})$  and then $f\in C([0,T];H^{-1})$. In particular 
$f(0)$ is well defined and $f(0)\in L^2$ since $f\in L^{\infty,2}$.
Then, again by Corollary \ref{c:UGamma} we have
\[
\| \partial_t \Ga f\|_\X \leq c( \|\partial_t f\|_{\tilde{\X}}  +\| f(0)\| )\leq c\| f\|_{\tilde{\Z}}.
\]
Notice also that we have
\beq\label{identity2}
\HH \Ga f = i (\partial_t \Ga f - f),
\eeq hence, 
\[
\| \HH \Ga f \|_{\infty,2} \leq \| \partial_t \Ga f  \|_{\infty,2} + \|f  \|_{\infty,2} \leq \|f\|_{\tilde{Z}},
\]
and the proof of \eqref{due} is complete.
\end{proof}
%\begin{remark}
%We stress that identities \eqref{identity1} and \eqref{identity2} will be needed in the sequel.
%\end{remark} 
 
\section{Well Posedness\label{s:wp}}
We want to study strong solutions of the Cauchy problem for the abstract NLS equation
\begin{equation} \label{cauchy}
\left\{\begin{aligned}
i \partial_t\psi (t)  &= {\HH} \psi (t) + F(\psi)(t) \\
 \psi (0)  &=  \psi_0\in \DD 
\end{aligned} \right.
\end{equation}
where $F(\psi) =\pm|\psi|^{p-1}\psi$.\\
By strong solution of \eqref{cauchy} we mean a function  $\psi\in C([0,T];\DD) \cap C^1([0,T];L^2)$ which satisfies the equation and the initial value as $L^2$ identities.\\
Through the Duhamel formula we replace the differential equation with its integral version.
More explicitly, we formulate the subsequent proposition. The proof straightforwardly follows the lines of the standard situation and we omit it (see section 4.1 in \cite{cazhar} for a detailed analysis in the abstract setting).
\begin{prop}%\label{equivalence}
A function $\psi\in C([0,T];\DD) \cap C^1([0,T];L^2)$ is a strong solution of \eqref{cauchy} if and only if it solves in $L^2(\RE^n)$ for every $t\in [0,T]$ the integral equation
\begin{equation*}%\label{integraleq}
\psi(t)=U(t)\psi_0-i\int_0^t\ U(t-s)F(\psi)(s)\ ds .
\end{equation*}
\end{prop}

We will often refer to the integral version in the following shortened form
\be \label{eq}
\psi=U\psi_0 -i \Ga F(\psi) ,
\ee
where $U$ and $\Gamma$ are defined in \eqref{evol} and \eqref{duham}.\\
The main result we want to prove in this section is 
\begin{teo} [Local Well-Posedness in $D(\HH)$]\label{local}
Assume that $p\geq 1$ if $n=2$ and $1\leq p<3/2$ if $n=3$ and let $\psi_0 \in \DD$. Then the following properties hold true.\\
1) There exists $T\in(0,1]$ and a solution of \eqref{eq}
in $C([0,T];\DD) \cap C^1([0,T];L^2)$.\\
2) The solution $\psi\in W^{1,r}((0,T);L^{p+1})$, where $r$ is such that $(r,p+1)$ is admissible (as in Def. \ref{d:rp}).\\
3) The solution enjoys unconditional uniqueness in $C([0,T];\DD)$.\\
4) There is continuous dependence on initial data, in the following sense. Let $\psi_0^n\to\psi_0$ in $\DD$; then denoted as $\psi$ and $\psi^n$ the solutions of \eqref{eq} corresponding to initial data $\psi_0$ and $\psi^n_0$, one has $\psi^n\to\psi$ in $C([0,T];\DD)$.\\
5) The following blow-up alternative holds. Let the maximal existence time be defined as 
\[T^*=\sup_{T>0}\left\{\psi\in C([0,T],\DD)\cap C^1([0,T], L^2) \ \ \text{solves}\ \ \eqref{eq} \right\};\]
then 
\[\lim_{t\to T^*} \| \psi(t)\|_{\DD}<\infty \ \ \Longrightarrow\  \ T^*=\infty .\]
\end{teo}
\begin{remark}\label{distributional}
Let be $\psi\in\DD(\HH)$ a solution of \eqref{cauchy}, and let be $\zeta\in C^{\infty}((0,T)\times \RE^n)$ a test function. The function $\psi$ can be considered a (regular) distribution. Testing against $\zeta$ the distribution
\begin{equation*}
i \partial_t\psi  + \Delta \psi   \mp|\psi|^{p-1}\psi  \\
\end{equation*}
and recalling that $\psi=\phi^\la   + q   \,  G^\lambda$ and $(-\Delta+\lambda)G^\la=\delta_0$, we obtain 
\begin{align*}
&\langle  i \partial_t\psi, \zeta\rangle  -  \langle\psi,(-\Delta+\lambda) \zeta\rangle +\lambda \langle\psi, \zeta\rangle \mp\langle |\psi|^{p-1}\psi,\zeta\rangle \\
= &\langle  i \partial_t\psi, \zeta\rangle  -  \langle  \phi^\la   + q   \,  G^\lambda,(-\Delta+\lambda) \zeta\rangle +\lambda \langle\psi, \zeta\rangle \mp\langle |\psi|^{p-1}\psi,\zeta\rangle\\
= &\langle  i \partial_t\psi, \zeta\rangle  -  \langle (-\Delta+\lambda) \phi^\la, \zeta \rangle  - \langle q\delta_0, \zeta\rangle +\lambda \langle\psi, \zeta\rangle \mp\langle |\psi|^{p-1}\psi, \zeta\rangle \\
= &\langle  i \partial_t\psi, \zeta\rangle  -  \langle (\HH+\lambda) \psi, \zeta \rangle  - \langle q\delta_0, \zeta\rangle +\lambda \langle\psi, \zeta\rangle \mp\langle |\psi|^{p-1}\psi, \zeta\rangle\\
= &\langle i\partial_t\psi -  \HH\psi  \mp |\psi|^{p-1}\psi, \zeta\rangle - \langle q\delta_0, \zeta\rangle=- \langle q\delta_0, \zeta\rangle\ .
\end{align*}
This means that a strong solution of \eqref{cauchy} solves as a distribution the NLS equation with a Dirac delta source
\begin{equation*}
i \partial_t\psi   =  -\Delta \psi   \pm|\psi|^{p-1}\psi - q\delta_0\ . \\
\end{equation*}
Notice that at this level the special form of $q=q(t)$ given by the boundary condition is not important.
\end{remark}

\begin{remark}%\label{3.3}
The case $p=1$ corresponds to the linear equation and it is well known. Hence in  the forthcoming analysis we shall always assume $p>1$.
\end{remark}
\begin{remark}%\label{r:3.4}
The presence of $T\in(0,1]$ in part $1)$ of the statement of the Theorem \ref{local} is only due to avoiding constants depending on the existence time in the many bounds appearing in the proof. This limitation is inessential as regards local existence (see also remark \ref{T_0}).
\end{remark}
\begin{remark} According to a usual and convenient strategy, we will prove existence and uniqueness of solution of the integral equation \eqref{eq} in weaker norms than the ones stated and then the further regularity will follow from the equation.
\end{remark}
We split the proof of the local well posedness for strong solutions in separate subsections.
\subsection{Local Existence and conditional uniqueness}
The first proposition collects some simple and useful properties of the nonlinearity $F$ used often in the subsequent analysis; the first two are well known, while the third is specific of the present problem.% a rephrasing of Sobolev-like embedding given in \eqref{DDalpha-embedding}.
\begin{prop}\label{p:banale} Let $q\geq p\geq 1$ and consider the map $v\mapsto F(v)=\pm |v|^{p-1}v$. Then the following holds true.\\
1) $F:L^q\to L^{q/p}$ is continuous and 
\[
\|F(v)\|_{q/p}\leq C\|v\|_q^p
\]
2) F is continuously differentiable in the real sense and its derivative at the point $v\in L^q$ is given by
\beq\label{derivata1}
F'(v)w=|v|^{p-1}w+(p-1)|v|^{p-3}v \Re(v \overline w)\qquad   \forall v, w\in L^q	 .
\eeq
Moreover the derivative map satisfies the bounds
\begin{equation}\label{derivata2}
\|F'(v)\|_{L^q\to L^{q/p}}\leq C\| v\|_q^{p-1} \qquad \text{and} \qquad \|F'(v)w\|_{q/p}\leq C\| v\|_q^{p-1}\|w\|_q\ .
\end{equation}
3) Let $p> 1$ if $n=2$, $1< p<3/2$ if $n=3$ and let $v \in \DD$. Then
\be \label{banale}
\| F(v) \| \leq c \| v \|^p_{\DD} .
\ee
%In particular, $F$ is Lipschitz on bounded sets as a map $F:L^q\to L^{q/p}$ and $F'(v)$ is bounded.
\end{prop}
\begin{proof}[{\bf Proof}]
The fact that $F:L^q\to L^{q/p}$ is an easy check, as it is formula \eqref{derivata1} by using the formula  $F'(v)w = \frac{d}{ds}F(v+sw)|_{s=0}$. Continuity and differentiability are well known properties of the Nemitskii operator $v\mapsto F(v)$ (see for example \cite{AM07}, Section 1.3 and \cite{K2}, Section 4).
Concerning 3), notice that, by the definition of $\DD$, $v = \phi^\lambda + \phi^\lambda(\0)\, G^\lambda$, hence 
\[
|F(v) | \leq c \big(|\phi^\la|^{p} + |\phi^\la (\0)|^{p} \, |G^\lambda|^{p} \big).
\]
By \eqref{r:embedding}, $\|\,|\phi^\lambda|^p\,\| = \|\phi^\lambda\,\|_{2p}^p  \leq c \| \fila \|_{H^2}^p$  and $\| \fila \|_{\infty} \leq c \| \fila \|_{H^2}$. Moreover 
$|G^\lambda|^{p}\in L^2(\RE^n)$ for the considered range of $p$. Hence, recalling that $\|v\|_{\DD} = \|\phi^\lambda\|_{H^2}$,  \eqref{banale} immediately
follows.
\end{proof}
\begin{remark}
We will often use property $2)$ of Proposition \ref{p:banale} in the case $q=p+1$, obtaining 
\[
\|F'(v)\|_{L^{p+1}\to L^{1+1/p}}\leq C\| v\|_{p+1}^{p-1}\ \ \ \ \text{and}\ \ \ \ \|F'(v)u\|_{1+1/p}\leq C\| v\|_{p+1}^{p-1}\|u\|_{p+1} .
\]
\end{remark}
\begin{remark}
Validity of Proposition \ref{p:banale} is not restricted to the pure power nonlinearity. If the function $F:\CO\to \CO$ defining the nonlinearity satisfies the bounds
\[%\label{growth}
|F(z)|\leq C|z|^p \ \ \ \ \text{and}\ \ \ \ \ |F'(z)|\leq C|z|^{p-1}\ \ \ \ \ \ \ \ \ z\in \CO 
\]
then Proposition \ref{p:banale} still holds true. The proof is analogous to the one given above for part 1) and for 2) and 3) see \cite{K2}, Section 4).
\end{remark}
\vskip5pt
\noindent
In the following Proposition recall that $r= \frac{4(p+1)}{n(p-1)}$ (see Def. \ref{d:rp}) and that $r>2$. 

\begin{prop} \label{cane} Assume that $p>1$ if $n=2$ and $1< p<3/2$ if $n=3$. Set $\beta = \frac2r$. We have $F:\Z \to \tilde\Z$ and for $T\leq 1$ there holds true:
\[%\label{virus}
\| F(v)-F(v)(0) \|_{\tilde \Z} \leq c T^{1-\beta} \|v\|^p_\Z \qquad \forall v\in \Z. 
\]
\end{prop}
\begin{proof}[{\bf Proof}]
We prove first that $\partial_tF(v) \in \tilde \X$. To this aim we shall prove that   
\[
F(v) \in W^{1,r'}((0,T);L^{1+1/p}),
\]
which implies the claim.  Note that $v\in \Z$ implies $v\in L^{\infty}([0,T],\DD)$, hence, by the embedding \eqref{DDalpha-embedding}, there holds $  \|v\|_{\infty, p+1} \leq c \|v\|_\Z$. Moreover, $\|F(v)\|_{1+1/p} = \|v\|_{p+1}^p$. If $p<r-1$, by H\"older inequality, one has  $\|F(v)\|_{r', 1+1/p}\leq T^{\frac{r-1-p}{r}}\|v\|_{r,p+1}^p$, while for $p\geq r-1$ one obtains $\|F(v)\|_{r', 1+1/p}\leq c \|v\|_\Z^{p-r+1}\|v\|_{r,p+1}^{r-1}$. So, $\|F(v)\|_{r', 1+1/p}\leq c \|v\|_\Z^{p}$, and $F(v)\in  L^{r'}((0,T);L^{1+1/p})$. To proceed, note that 
 \begin{equation}\label{harbour}
  | F(v)(t',x)  -  F(v)(t,x)|  \leq c ( |v(t',x)|^{p-1} + |v(t,x)|^{p-1} )\, | v(t',x) - v(t,x)|.
 \end{equation}
 Taking into account that $\| f^{p-1} g\|_{1+1/p} \leq \|f\|_{p+1}^{p-1}\|g\|_{p+1}$, we have
\begin{align*}
\big\| F(v)(t') - F(v)(t)\big\|_{1+1/p} & \leq c \big(  \|v(t')\|_{p+1}^{p-1} +  \|v(t)\|_{p+1}^{p-1}\big)\, \big\| v(t') - v(t)\big\|_{p+1} .
\end{align*}
Since $ \big\| v(t') - v(t)\big\|_{p+1} \leq \int_t^{t'} \|\partial_s v(s) \|_{p+1} ds$, setting $\varphi(s) = c \|v\|_{\infty,p+1}^{p-1}  \|\partial_s v(s) \|_{p+1}$ (for some constant $c$ large enough) one has 
\[
\big\| F(v)(t') - F(v)(t)\big\|_{1+1/p}  \leq \int_t^{t'} \varphi(s) ds
\]
for almost all $t,t' \in [0,T]$. By Theorem 1.4.40 in \cite{cazhar} it follows that $F\in W^{1,r'}((0,T);L^{1+1/p})$ and  $\|\partial_t F(v)\|_{r',1+1/p} \leq \| \varphi \|_{L^{r'}(0,T)}$. Additionally, by H\"older inequality in time, 
\[  \| \varphi \|_{L^{r'}(0,T)} \leq c  \|v\|_{\infty,p+1}^{p-1} T^{1-\beta}   \|\partial_t v \|_{r,p+1}\leq  c   T^{1-\beta}   \|v \|_{\Z}^p,\]
and  
\[
\big\| \partial_t \big(F(v)-F(v(0))\big)\big\|_{\tilde{\X}} =   \| \partial_t F(v)\|_{\tilde{\X}} \leq \|\partial_t F(v)\|_{r',1+1/p} \leq c   T^{1-\beta}   \|v \|_{\Z}^p. 
\]
 
Note that $\|F(v)\|_{\infty,2} \leq c \|v\|_\Z^p$ by Prop. \ref{p:banale}. Hence, $F : \Z \to \tilde \Z$. 

Next we prove that 
\begin{equation}\label{wild}
\| F(v)-F(v)(0) \|_{\infty,2} \leq c T^{1-\beta} \|v \|_\Z^p.
\end{equation}

By interpolation (see Eq. \eqref{interpolationgeneral}, with $a=0$, $b=1$, $\theta = s/2$, $s\in (0,2)$), it follows that 
\[
\| v(t) -v(t')\|_{\DD^{\frac{s}{2}}} \leq c   \| v(t) -v(t')\|_{\DD}^{\frac{s}{2}} \, \| v(t) -v(t')\|^{1-\frac{s}{2}},
\]
 for all $s\in(0,2)$. Since $\partial_t v\in L^{\infty,2}$, there holds true   $\| v(t) - v(t') \| \leq |t-t'| \|\partial_t v \|_{\infty,2}$ for all $t,t'\in[0,T]$. Moreover,  $v\in L^\infty ([0,T];\DD)$, hence, after a possibile modification on a set of measure zero,  $v$ is a bounded  mapping of $[0,T] \to \DD$ which satisfies the inequality 
\begin{equation}\label{notworried}
\| v(t) -v(t')\|_{\DD^{\frac{s}{2}}}  \leq c  |t-t'|^{1-\frac{s}{2}}  \| v \|_{\infty, \DD}^{\frac{s}{2}} \|\partial_t v \|_{\infty,2}^{1-\frac{s}{2}}\leq  c  |t-t'|^{1-\frac{s}{2}}  \| v \|_{\Z} \qquad s\in(0,2).
\end{equation}
The latter inequality implies that the  representative of $v\in \Z$ which is a bounded map from  $[0,T] \to \DD$ is a H\"older continuous map from  $[0,T] \to \DD^{\frac{s}2}$;  from now on we denote by $v(0)$ its  value  in $t=0$. The function $v(0)$ can be understood as function in $L^\infty([0,T],\DD^{\frac{s}2})$, independent on $t$.  
%Note moreover that $v\in L^2 ((0,T);\DD)$ and  $\partial_t v \in L^2((0,T);L^2)$. Hence, by Theorem 3.1 in \cite{LM1} (here employed with $X$ and $Y$ in \cite{LM1} given by $X = \DD \equiv \DD^1$ and $Y = L^2(\RE^n) \equiv \DD^0$), together with the obvious interpolation $\DD^{\frac12} =  [\DD^1, \DD^0]_{1/2}$,  it follows that, after a possibile modification on a set of measure zero,  $v$ is a continuous mapping of $[0,T] \to \DD^{\frac12}$, i.e., $v \in C^0([0,T];\DD^{\frac12})$. Additionally, due to the embedding $\DD^{\frac{1}{2}} \hookrightarrow\DD^{\frac{s}{2}}$, $s\in[0,1]$, it holds true   $v \in C^0([0,T];\DD^{\frac{s}{2}})$, and $v(0)$ is well defined as a function in $\DD^{\frac{s}{2}}$. \\
Setting $t' = 0$ in Eq. \eqref{notworried},  and taking the essential supremum we infer  
\begin{equation*}%\label{matita2}
\| v -v(0)\|_{\infty, \DD^{\frac{s}{2}}} \leq    c T^{1-\frac{s}2}\|v\|_{\Z} \qquad s\in(0,2).
\end{equation*}
Next we use the embeddings $\DD^{\frac{s}{2}} \hookrightarrow H^s \hookrightarrow L^{2p}$, see Eq. \eqref{modifiedembeddings}, with  $s = s_c(2p) = \frac{n(p-1)}{2p}$.    To proceed,  recall the inequality  \eqref{harbour} and 
notice that, by H\"older inequality,  it follows that 
\begin{equation}\label{semplice0}
\| f^{p-1} g\| \leq \|f\|^{p-1}_{2p} \|g\|_{2p}.
\end{equation}
Hence, 
\[\begin{aligned}
\| F(v) - F(v)(0) \|_{\infty,2} \leq  &c \big( \|v\|_{\infty, 2p}^{p-1} + \|v(0)\|_{\infty,2p}^{p-1} \big) \|v-v(0)\|_{\infty, 2p}
 \\ 
 \leq &  c \big( \|v\|_{\infty, \DD^{\frac{s}2}}^{p-1} + \|v(0)\|_{\infty,\DD^{\frac{s}2}}^{p-1} \big) \|v-v(0)\|_{\infty,\DD^{\frac{s}2}}\leq  c T^{1-\frac{s}2}\|v\|_{\Z}^p \leq c T^{1-\frac{2}r}\|v\|_{\Z}^p,
\end{aligned}\]
where in the latter inequality we used $1- \frac{s}2 > 1 - \frac{2}r $. This concludes the proof of inequality \eqref{wild}.
\end{proof}

\begin{remark}As pointed out in the proof of Prop. \ref{cane}, if $v \in \Z$ then one has $v \in C^{0,1-\frac{s}{2}}([0,T], \DD^{\frac{s}{2}})$, i.e., $v$ is a H\"older continuous map from $[0,T] \to \DD^{\frac{s}{2}}$. More precisely, there holds true 
\[v \in \Lip([0,T]; L^2) \cap C^{0,1-\frac{s}{2}} ([0,T];\DD^{\frac{s}{2}}) \subset \Lip([0,T]; L^2) \cap C^{0,1-\frac{s}{2}} ([0,T];L^{2p}).\]
 As a consequence of inequalities \eqref{harbour} and \eqref{semplice0} one has $F(v) \in C^{0,1-\frac{s}{2}} ([0,T];L^{2})$. 
\end{remark}

For $v\in \Z$ and $\psi_0\in\DD$ let us define the map
\begin{equation}\label{Phi}
\Phi(v) = U\psi_0 -i\Ga F(v).
\end{equation}
and set $ B_\Z(R) = \{ v\in \Z \, | \, \|v\|_\Z \leq R\}$.
\begin{prop}
Let $\psi_0\in \DD$ and define 
\[ 
\E=\{ v\in B_\Z(R) |\, v(0)= \psi_0\}. 
\]
Then:\\
1) $\E$ is a complete metric space with respect to the metric induced by the $\X$-norm.\\
2) Assume that $p>1$ if $n=2$ and $1< p<3/2$ if $n=3$. 
Then there exist $R$ big enough and $T$ sufficiently small such that $\Phi:\E \to \E$.
%Moreover $\E$ is a complete metric space with respect to the metric induced by the $\X$-norm.
\end{prop}
\begin{proof}[{\bf Proof}]
Let us prove that $\E$, with the metric induced by the $\X$-norm, is a complete metric space. One has $\E\subset\Z\subset\X$.  $\X$ is a Banach space, hence, to prove that it $\E$ is complete it is enough to prove that it is a closed subset in $\X$.  Let $u$ be a limit point of $\E$ so that there exists $\{ u_n \}$ with $u_n \in \E$ and $\|u_n - u\|_\X \to 0$; we want to prove that $u\in \E$.  Obviously, $u\in\X$ and $\|u\|_\X \leq R$. We are left to prove that $\|\partial_t{u}\|_\X\leq R$,  $\|\HH u\|_{\infty,2}\leq R$, and $u(0) = \Psi_0$. We recall that $\partial_t u$ is defined as a distribution on the test functions $\varphi\in C^{\infty}_0((0,T))$ by $\langle \varphi, \partial_t u \rangle = -\langle \partial_t \varphi, u \rangle$. Notice that by hypothesis we have   $ u_n \in W^{1,r}((0,T), L^{p+1}) $ and $\|\partial_t u_n\|_{r,p+1}\leq R\ .$ By well known properties of vector valued Sobolev spaces (see for example Corollary 1.4.42 in \cite{cazhar}) one concludes that there exists a subsequence $u_{n_k}\rightharpoonup v\in  W^{1,r}((0,T), L^{p+1})$ and  $\|\partial_t v\|_{r,p+1}\leq \liminf \|u_{n_k}\|\leq R\ $. Finally by uniqueness $u=v \ .$
%La dimostrazione continua da: In order to prove that $\| \dot u \|_{\infty,2} \leq R$
%Let $\{ u_n \}$ be  such that $u_n \in \E$ and $\|u_n - u\|_\X \to 0$, we want to prove that $u\in \E$.  
%Since $\| \dot{u}_n\|_{ {r, p+1}}  \leq R$ then there exists $v \in  $ such that $\dot u_n \rightharpoonup v$ and $\| v\|_{r, p+1} \leq R$.
%It is straightforward to prove that $\dot u =v$ integrating against a space-time test function $\varphi$. Indeed we have
%$\lan \varphi , \dot u_n \ran  \to \lan \varphi, v\ran$ and $\lan \varphi , \dot u_n \ran = -\lan \dot \varphi , u_n \ran \to -\lan \dot \varphi , u\ran$.
In order to prove that $\partial_t u\in L^{\infty, 2}$ and $\| \partial_t u \|_{\infty,2} \leq R$, the same reasoning works replacing weak convergence with weak-* convergence and invoking again Corollary 1.4.42 in \cite{cazhar}. 
%\[
%\lan \varphi , \dot u \ran =  - \lan \dot \varphi , u_n \ran - \lan \dot \varphi , u -u_n \ran =  - \lan \varphi , \dot  u_n \ran - \lan \dot \varphi , u -u_n \ran.
%\]
%Since $\| u_n - u \|_{\infty,2} \to 0 $ then $ \lan \dot \varphi , u -u_n \ran \to 0 $. Moreover $ | \lan \varphi , \dot  u_n \ran  | \leq \| \varphi \|_{1,2} \| \dot  u_n\|_{\infty,2} \leq 
% \| \varphi \|_{1,2} \, R$. Therefore we have that $ | \lan \varphi , \dot u \ran  | \leq  \| \varphi \|_{1,2} \, R$ and then $\|\dot u \|_{\infty,2} \leq R$ (because $L^{\infty,2}$ is the dual of $L^{1,2}$).\footnote{nota: qui sto usando che il duale di $L^1$
%\`e $L^\infty$ }.
Now we prove that $\| \HH u \|_{\infty,2} \leq R$. Since $\| \HH u_n (t) \| \leq R$ for a.e. $t\in [0,T]$ then there exists $v(t)\in L^2$ and a subsequence that we denote with $\HH u_{n_k}(t)$ such that $\HH u_{n_k} (t) \rightharpoonup v(t)$ a.e. in $[0,T]$  and $\|v(t)\|\leq R$. Recall that $C^\infty_0(\RE^n\setminus\{0\})$ is a dense subset in $L^2(\RE^n) \ \, n=2,3$ and let  $\varphi\in C^\infty_0(\RE^n\setminus\{0\})$. We have $\lan \varphi , \HH u_{n_k}(t) \ran \to  \lan \varphi , v(t) \ran$, where now $\lan\ ,\ran$ is the $L^2$ scalar product.
Moreover 
\[
\lan \varphi , \HH u_n(t) \ran = \lan \HH \varphi ,  u_n(t) \ran \to \lan \HH \varphi ,  u(t) \ran = \lan \varphi , \HH u(t) \ran
\]
since $u(t)\in \DD$ a.e. in $[0,T]$. Then $\HH u(t)= v(t)$ and $\|\HH u (t)\| \leq R$ a.e. in $[0,T]$.\\
We are left to prove that $u(0) = \psi_0$ as an $L^2$ identity. We know that $u_n(t)\in W^{1,\infty}((0,T);L^2)$ and $u_n(0)=\psi_0$.  Being $ W^{1,\infty}((0,T);L^2) \hookrightarrow C([0,T]; L^2)$ one has that $u_n$ converges in $C([0,T];L^2)$ and $u(0)=\psi_0\ $ as an identity in $L^2$.\\
Now we prove that $\Phi:\E \to \E$. Notice that, by Proposition \ref{gatto}, $U\psi_0 \in \Z$. Moreover since $F(\psi_0)\in L^2$ (by Proposition \ref{p:banale}), we can say $F(\psi_0)\in \tilde{\Z}$ since it depends on $t$ in a trivial way; therefore $\Ga F(\psi_0)\in \Z$
by Proposition \ref{gatto}. 
We  choose $R> \| U\psi_0 \|_\Z + \|\Ga  F(\psi_0)\|_\Z$ such that $\E$ is not empty since $ U\psi_0\in \E$.
Adding and subtracting $\Gamma F(\psi_0)= \Gamma F(v)(0)$ to the r.h.s. of Eq. \eqref{Phi}, we have  by Proposition \ref{cane}
\begin{align*} %\label{sgombro}
\| \Phi(v) \|_\Z & \leq  \| U\psi_0 \|_\Z + \|\Ga F(\psi_0)\|_\Z +\| \Gamma (F(v) - F(v)(0))\|_{\Z} \\ 
& \leq  \| U\psi_0 \|_\Z + \|\Ga F(\psi_0)\|_\Z +\| F(v) - F(v)(0)\|_{\tilde{\Z}} \leq  \| U\psi_0 \|_\Z + \|\Ga F(\psi_0)\|_\Z + c T^{1-\beta} R^p
\end{align*}
Then $\Phi:\E \to \E$ if $T$ is sufficiently small.
\end{proof}
Now we can prove part $1)$ and $2)$ of Theorem \ref{local}.

\begin{proof}[{\bf Proof of Theorem \ref{local}. Parts $1)$ and $2)$}] 
For sufficiently small $T$, $\Phi$ is a contraction in the $\X$-norm. Indeed by \eqref{lampada}, we have
\[
\| \Phi(u) - \Phi(v)\|_\X= \| \Ga( F(u) - F(v))\|_\X \leq  c \|  F(u) - F(v) \|_{\tilde{\X}}.
\]
Notice that
\[
\lf| ( F(u) - F(v))(t,x) \ri| \leq c ( |u(t,x)|^{p-1} + |v(t,x)|^{p-1} )\, | u(t,x) - v(t,x)|.
\] 
Using the inequality  $\| f^{p-1} g\|_{1+1/p} \leq \|f\|_{p+1}^{p-1}$, we have
\begin{align*}
\| ( F(u) - F(v))(t) \|_{1+1/p} & \leq c (  \|u(t)\|_{p+1}^{p-1} +  \|v(t)\|_{p+1}^{p-1})\, \| u(t) - v(t)\|_{p+1} .
%\\
%  & \leq c (  \|u\|_{\infty, p+1}^{p-1} +  \|v\|_{\infty, p+1}^{p-1})\, \| u(t) - v(t)\|_{p+1}
\end{align*}
Therefore we have
\be\label{pinco2}
\|  F(u) - F(v) \|_{r', 1+1/p} \leq T^{1-\beta} \|  F(u) - F(v) \|_{r , 1+1/p} \leq 
 c  T^{1-\beta} (  \|u\|_{\infty, p+1}^{p-1} +  \|v\|_{\infty, p+1}^{p-1})\, \| u - v\|_{r, p+1}
\eeq
with $\beta = 2/r$ (as in Proposition \ref{cane}). Hence, by the embedding \eqref{DDalpha-embedding},   for $u,v \in \E$ we obtain 
\beq\label{contrazione}
\| \Phi(u) - \Phi(v)\|_\X \leq  c  \, T^{1-\beta} \, R^{p-1}  \| u - v\|_\X.
\eeq
For sufficiently small $T$, $\Phi$ is a contraction in  the $\X$-norm. Since $\E$ is complete with respect to the metric induced by the $\X$-norm then the fixed point equation
\[
u = \Phi(u)
\]
admits a solution $\psi \in \E$. In particular, $\psi \in L^{\infty}([0,T];\DD)$, $\partial_t \psi \in L^\infty([0,T];L^2)$, $\psi\in W^{1,r}((0,T);L^{p+1})$ and  $\psi$ satisfies the identity $\psi=U\psi_0 -i \Ga F(\psi) $. We are left to prove that $\psi  \in C([0,T];\DD) \cap C^1([0,T];L^2)$. Obviously, $U\psi_0$ has the required properties, since $\psi_0\in \DD$ and thanks to the properties of the linear evolution. Concerning $\Gamma F(\psi)$, we start by noticing that $F(\psi)(t) \in L^2$ (a.e. in $[0,T]$), by Proposition \ref{p:banale}, hence $F(\psi) \in L^{\infty,2}$ and finally $U(t-\cdot)  F(\psi) (\cdot)\in L^1((0,T); L^2)$. So, by absolute continuity of the integral, one concludes $\psi\in C([0,T];L^2)$. From the Duhamel formula it is also immediate that $\Gamma F(\psi)\in C([0,T];L^2)$. By the identity (see also \eqref{identity1})
   \[ 
\partial_t\Gamma F(\psi) (t) =   UF(\psi) (0)   + \int_0^t U(t-s) \partial_s F(\psi) (s)  \, ds.
\]
using Proposition \ref{propa} e) and taking into account that $\partial_s F(\psi)(s)\in L^{r',1+1/p}$ we have that $\partial_t\Gamma F(\psi) (t) \in L^{\infty,2}$. On the other hand it is well known that it actually holds the stronger result $\Gamma v\in C([0,T], L^2)$ for $v\in L^{r',1+1/p}$ with $(r,p+1)$ admissible (see Remark \ref{strengthening}). Finally, exploiting the fact that $\HH$ is the infinitesimal generator of $U(t)$ (see also \eqref{identity2})  we have the identity
\[ 
\HH\Gamma F(\psi) (t) = i\partial_t\Gamma F(\psi) - i F(\psi)\ .
\]
The r.h.s. belongs to $C([0,T];L^2)$, or equivalently $\Gamma F(\psi)  \in C([0,T];\DD).$
%\[ 
%\HH\Gamma F(\psi) (t) = \int_t^{t'} U(t'-s) \HH  F(\psi) (s)   ds - 
%\int_0^t (U(t-s)-U(t'-s)) \HH  F(\psi) (s)   ds.
%\]
%Hence, 
%\[ 
%\|\HH\Gamma F(\psi) (t') -\HH\Gamma F(\psi) (t) \|  \leq  \int_t^{t'} \|\HH \Gamma F(\psi) (s) \| ds
%+ \int_0^t \|(U(t-s)-U(t'-s)) \HH  F(\psi) (s)\|   ds.
%\]
%Since  the r.h.s. obviously converges to zero as $t'\to t$, $\Gamma F(\psi) (t) \in C([0,T];\DD) $. Similarly,  
%   \[ 
%\partial_t\Gamma F(\psi) (t) =   F(\psi) (t)   + \int_0^t U(t-s) \,\HH F(\psi) (s)  \, ds.
%\]
%tells us that $\Gamma F(\psi) (t) \in C^1([0,T];L^2)$.
\end{proof}
%Notice that we proved also that $\dot \psi \in L^{r,p+1}$. \diego{Toglierei questa osservazione.}

\begin{cor} [Local well-posedness for strong solutions]\label{strongeq} 
Let $p>1$ if $n=2$ and $1<p<3/2$ if $n=3$. For any $\psi_0 \in D({\HH})$ there exists $T \in (0, + \infty)$ s.t. the
initial value problem \eqref{cauchy}
%\begin{equation} \label{cauchy}
%\left\{  \begin{aligned}
%i \partial_t\psi (t) & = {\HH} \psi (t) \pm |
%\psi (t) |^{p-1}  \psi (t) \\
%\psi (0) & =    \psi_0  
%\end{aligned}\right.
%\end{equation}
has a unique solution $\psi \in C ( [0,T) , \DD) \cap C^1 ( [0,T), 
L^2 (\erre))$.
\end{cor}

\begin{remark}
Notice that for a strong solution of the equation $\psi\in C([0,T),\DD)\cap C^1([0,T), L^2)$, the existence time given in the local well-posedness Theorem actually depends only on $\|\psi_0\|_{\DD}$. In fact, the $L^2$-norm of $\partial_t\psi$ can be bounded in terms of the graph norm of $\psi$ just taking into account that equation in \eqref{cauchy} holds as an $L^2$ identity and using the estimate \eqref{banale}.
\end{remark}

\subsection{Unconditional Uniqueness. Proof of Theorem \ref{local}. Part $3)$}
In the proof of local existence, the fixed point technique guarantees uniqueness only for those solutions $\psi\in C([0,T], \DD)$ that belong to the auxiliary space $L^{r,p+1}$. In this paragraph we show that actually the latter condition is not needed. 
\begin{prop}
Assume that  $p> 1$ if $n=2$ or $1< p<\frac{3}{2}$ if $n=3$. Take $\psi_0 \in \DD$. If $\psi_1$  and $\psi_2$ are in $L^\infty([0,T];\DD)$ for some $T>0$ and  are two solutions of Eq. \eqref{eq}, then $\psi_1= \psi_2$. 
\end{prop}
\begin{proof}[{\bf Proof}] %The case $p=1$ is trivial since it corresponds to the linear equation. In what follows we discuss the case $p>1$. We borrow some ideas from  the proof of Prop. 4.2.1 in \cite{caz}.  
Let $\tau$ be any time in $(0, T]$. 
%($\tau\leq 1$ is needed because we shall use Corollary  \ref{c:UGamma}). We start with the elementary inequality 
%\begin{equation}\label{trivial}
%\big||\psi_1|^{p-1}\psi_1 -  |\psi_2|^{p-1}\psi_2 \big| \leq C \big(|\psi_1|^{p-1} + |\psi_2|^{p-1}\big) |\psi_1 - \psi_2|.
%&\end{equation}
%Hence, by H\"older's inequality,
%\[
%\big\||\psi_1|^{p-1}\psi_1 -  |\psi_2|^{p-1}\psi_2 \big\|_{L^{1+\frac1p}(\RE^n)} \leq C \big(\|\psi_1\|_{L^{p+1}(\RE^n)}^{p-1} + \|\psi_2\|_{L^{p+1}(\RE^n)}^{p-1}\big) \|\psi_1 - \psi_2\|_{L^{p+1}(\RE^n)}.
%\]
Reasoning as in the derivation of \eqref{pinco2} we obtain
\[ \begin{aligned}
&\big\| F(\psi_1) -  F(\psi_2) \big\|_{L^{r'}([0,\tau],L^{1+\frac1p})} 
 \leq   C \big( 
 \| \psi_1 \|^{p-1}_{L^\infty([0,\tau],L^{p+1})} +  \| \psi_2 \|^{p-1}_{L^\infty([0,\tau],L^{p+1})}\big) \|\psi_1 - \psi_2\|_{L^{r'}([0,\tau],L^{p+1})} 
\end{aligned}
\]
with  $r = \frac{4(p+1)}{n(p-1)}$ as in Definition \ref{d:rp} so that  we can apply Proposition \ref{propa} (for any $T>0$, see Remark \ref{T_0}).  

By Eq. \eqref{eq}, 
\[
|\psi_1 - \psi_2| =\big| \Gamma(F(\psi_1) -  F(\psi_2) )\big|.
\]
Hence, using Prop. \ref{propa}.$f)$ and  the inequality above, we infer
\begin{align}
\|\psi_1 - \psi_2\|_{L^{r}([0,\tau],L^{p+1})}
\leq & C  \|F(\psi_1) - F(\psi_2) \|_{L^{r'}([0,\tau],L^{1+\frac1p})}  \nonumber \\ 
 \leq &C\big( 
 \| \psi_1 \|^{p-1}_{L^\infty([0,\tau],L^{p+1})} +  \| \psi_2 \|^{p-1}_{L^\infty([0,\tau],L^{p+1})}\big) \|\psi_1 - \psi_2\|_{L^{r'}([0,\tau],L^{p+1})} .  \label{factory}
\end{align}
%Recall that, by Sobolev embedding theorem (see, e.g., \cite[Th. 5.4]{Adams}), it holds true: 
%\begin{equation}\label{moreembeddings}
%\begin{aligned}
%& H^2(\RE^3) \hookrightarrow H^1(\RE^3) \hookrightarrow L^s(\RE^3) \qquad & 2\leq s \leq 6 ; \\ 
%& H^2(\RE^2) \hookrightarrow H^1(\RE^2) \hookrightarrow L^s(\RE^2)  & 2\leq s < \infty ;  
%\end{aligned}
%\end{equation}
%moreover, 
%\[
%H^2(\RE^n) \hookrightarrow C_B(\RE^n) \qquad n =2,3
%\]
%(here $C_B(\RE^n)$ is the Banach space of bounded continuous functions in $\RE^n$). \\

Since, by assumption, $\psi_1, \psi_2 \in L^\infty([0,\tau];\DD)$, the embedding \eqref{DDalpha-embedding}, together with  the inequality \eqref{factory}, give 
\begin{equation}\label{uniq-main-ineq}
\|\psi_1 - \psi_2\|_{L^{r}([0,\tau],L^{p+1})}
\leq C
\|\psi_1 - \psi_2\|_{L^{r'}([0,\tau],L^{p+1})}  \qquad   \forall \tau\in(0,T].
\end{equation}

Let $\phi(t) : = \|\psi_1(t) - \psi_2(t)\|_{L^{p+1}}$. By the inequality above,  together with H\"older's inequality,  we infer 
\[
\|\phi \|_{L^{r}([0,\tau_*])}
\leq C  \tau_*^{1-\frac2r}
\|\phi \|_{L^{r}([0,\tau_*])} .
\]
Since $r>2$, for $\tau_*$ small enough (such that $C  \tau_*^{1-\frac2r}<1$),  the latter inequality implies $\phi(t) =0$ a.e. in $[0,\tau_*]$. Next, assume that $\phi(t) =0$ a.e. in $[0,k\tau_*]$ for some positive integer $k$, then, inequality \eqref{uniq-main-ineq} (applied for $\tau = (k+1)\tau_*$) is equivalent to 
\[
\|\phi \|_{L^{r}([k\tau_*,(k+1)\tau_*])}
\leq C  
\|\phi \|_{L^{r'}([k\tau_*,(k+1)\tau_*])} .
\]
Hence, by using again H\"older's inequality, we infer  $\phi(t) =0$ a.e. in $[0,(k+1)\tau_*]$. We proceed in this way, by induction,  until  $(k_*+1)\tau_*\geq T$, for some positive integer $k_*$. In the final step we  address  the interval  $[k_*\tau_*,T]$. In this way we prove  $\phi(t) =0$ a.e. in $[0,T]$, henceforth $\psi_1 = \psi_2$ a.e. 
This concludes the proof of the proposition. 
%Assume that  $T>1$. By the discussion above, there exist $\ve\geq0$ arbitrarily small (possibly equal to zero) such that $\psi_1|_{t=1-\ve} = \psi_2|_{t=1-\ve}$. Since the equation is invariant by time translation the argument above does not depend on the choice of the initial time. In particular, two solutions that coincide for $t = 1-\ve$ must coincide a.e. in the interval $[1-\ve, \min\{2-\ve,T\}]$. Iterating this argument it is possible to prove uniqueness in an interval of the form $[0,\min\{k - \sum_{j=1}^{k-1}\ve_j,T\}]$. Since $\ve_j$ are arbitrarily small, after a finite number of steps the latter interval  will coincide with   $[0,T]$. 
\end{proof}

\subsection{Continuous dependence on initial data} %To end the proof we consider continuous dependence on initial data. 
\begin{proof}[{\bf Proof of Theorem \ref{local}. Part $4)$}]
Assume that $\|\psi_0-\psi_0^n\|_{\DD}\rightarrow 0$; let $\psi$ be the solution corresponding to the  initial datum $\psi_0$ and $\psi^n$ the solution corresponding to $\psi^n_0\ $ according to the local existence result proved in the previous section. Notice preliminarily that by hypothesis we have $\|\psi^n_0\|_\DD\leq 2\|\psi_0\|_\DD$ and from the local existence we obtain that there exists a time $T=T(\|\psi_0\|_\DD)$ and $n_0$ such that 
both $\psi$ and $\psi_n$ are defined in $[0,T] $ for $n\geq n_0$; moreover the following uniform bound holds 
\beq\label{uniform3}
\|\psi\|_{\infty,\DD}+\|\psi^n\|_{\infty,\DD}\leq C\|\psi_0\|_\DD\ .
\eeq
From \eqref{eq} and the analogous
\[
\psi^n=U\psi^n_0 -i \Ga F(\psi^n) \ 
\]
we obtain 
\[
\psi-\psi^n=U(\psi_0-\psi^n_0) -i(\Ga F(\psi) -\Ga F(\psi^n)).
\]
From Strichartz estimates and contractivity in the $\X$- norm of $\psi\mapsto \Gamma F(\psi)$ given in \eqref{contrazione} it follows that, choosing possibly a $T'<T$,  

\begin{equation*}
\|\psi-\psi^n\|_{\X}\leq C\|\psi_0-\psi^n_0\|  + \frac{1}{2}\|\psi -\psi^n\|_{\X}\leq C\|\psi_0-\psi^n_0\|_{\DD}  + \frac{1}{2}\|\psi -\psi^n\|_{\X}
\end{equation*}
and hence
\[
\|\psi-\psi^n\|_{\X}\leq 2C\|\psi_0-\psi^n_0\|_{\DD} .
\]
This gives continuity of the solution map in $\X$ and in particular in $L^r((0,T'),L^{p+1})$. 
Let us show that we also have $\|\partial_t\psi-\partial_t\psi^n\|_{r,p+1}\leq C\|\psi-\psi^n_0\|_\DD$, so that the solution map $\psi_0\to \psi(t, \psi_0)$ is continuous as a map from $\DD$ to $W^{1,r}((0,T'),L^{p+1})$ for suitable $T'\leq T$. Taking the time derivative of the integral equation both for $\psi$ and $\psi^n$, subtracting and rearranging we obtain 
\begin{align}\label{pinco}
\partial_t(\psi-\psi^n)=-i\Ga F'(\psi^n)\partial_t(\psi-\psi^n) +\mathcal R_1 +\mathcal R_2 
\end{align}
where 
\begin{align*}
\mathcal R_1&=-i U(\HH(\psi_0-\psi^n_0))-iU(F(\psi_0)-F(\psi^n_0))\\
\mathcal R_2&= - i\Ga \big(F'(\psi)-F'(\psi^n)\big)\partial_t\psi.
\end{align*}
By means of dispersive estimates \ref{propa}-$b)$ and \ref{propa}-$f)$ on $(0,T')$ we obtain 
\begin{align*} 
\|\partial_t(\psi-\psi^n)\|_{L^r((0,T'),L^{p+1})}  \leq &C(T')\| F'(\psi^n)\partial_t(\psi-\psi^n)\|_{L^{r'}((0,T'),L^{1+1/p})} \\ &  + C(T') \big(\|\HH (\psi_0-\psi^n_0)\|  + \|F(\psi_0)-F(\psi^n_0)\|\big)\\ & + C(T')\| \big(F'(\psi)-F'(\psi^n)\big)\partial_t\psi\|_{L^{r'}((0,T'),L^{1+1/p})} 
\end{align*}
where $ C (T')$ is a constant which is uniformly bounded for $T'\in (0,1]$.  Let us consider the first addendum in the previous inequality. By the bound in Eq. \eqref{derivata2} we have 
 \begin{equation*}%\label{temporary}
 \| F'(\psi^n(t))\partial_t(\psi-\psi^n)(t)\|_{1+1/p} \leq C \|\psi^n(t)\|_{p+1}^{p-1}\|\partial_t(\psi-\psi^n)(t)\|_{p+1}.
 \end{equation*}
Hence, by H\"older inequality in time, 
\[
\| F'(\psi^n)\partial_t(\psi-\psi^n)\|_{L^{r'}((0,T'),L^{1+1/p})}\leq  C T'^{ 1-\frac2r} \| \partial_t(\psi-\psi^n)\|_{L^{r}((0,T'),L^{p+1})} 
\]
where the bound is uniform in $n$ thanks to \eqref{uniform3} and embedding \eqref{DDalpha-embedding}. Taking a smaller $T'$ if needed, one gets
\begin{align*} 
\|\partial_t(\psi-\psi^n)\|_{L^r((0,T'),L^{p+1})}  \leq & \frac{1}{2} \| \partial_t(\psi-\psi^n)\|_{L^{r}((0,T'),L^{p+1})} \\ &+ C(T') \big(\|\HH (\psi_0-\psi^n_0)\|  + \|F(\psi_0)-F(\psi^n_0)\|\big) \\ &+ C(T')\| \big(F'(\psi)-F'(\psi^n)\big)\partial_t\psi\|_{L^{r'}((0,T'),L^{1+1/p})} 
\end{align*}
and hence
\begin{align*} 
\|\partial_t(\psi-\psi^n)\|_{L^r((0,T'),L^{p+1})}  \leq & 2C(T') \big(\|\HH (\psi_0-\psi^n_0)\|  + \|F(\psi_0)-F(\psi^n_0)\|\big) \\ &+ 2C(T')\| \big(F'(\psi)-F'(\psi^n)\big)\partial_t\psi\|_{L^{r'}((0,T'),L^{1+1/p})} .
\end{align*}
We have to show that the three terms on the r.h.s vanish when $\|\psi-\psi^n \|_{\DD}\to 0.$ For the first term this is obvious. For the second term we recall (see Proposition \ref{p:banale}-1)) that the map $F:L^{2p}\to L^2$ is continuous and  
\[
\|F(\psi_0^n)-F(\psi_0)\|\leq C\|\psi^n_0-\psi^n_0\|_{2p}^p\leq C\|\psi^n_0-\psi_0\|_{\DD}^p
\]
where the last inequality follows from \eqref{DDalpha-embedding}.
One concludes that $\|F(\psi_0)-F(\psi^n_0)\|\to 0$ as $\|\psi_0-\psi_0^n\|_{\DD}\to 0$. For the last term, exploiting again Proposition \ref{p:banale}-2),  in particular the continuity of $F'$,  one has that $\| \big(F'(\psi)-F'(\psi^n)\big)\partial_t\psi\|_{1+1/p}\to 0$ point-wise a.e. in time. Moreover, notice that
\[
\| \big(F'(\psi)-F'(\psi^n)\big)\partial_t\psi\|_{1+1/p} \leq C(\|\psi\|_{p+1}^{p-1}+ \|\psi^n\|_{p+1}^{p-1})\|\partial_t\psi\|_{p+1}\leq C\|\partial_t\psi\|_{p+1}
\]
where the latter bound  is uniform in $n$ thanks again to \eqref{DDalpha-embedding} and \eqref{uniform3}. Now we know from local existence part that $\partial_t\psi\in L^{r,p+1}$, and being $r'<r$, it also holds $\partial_t\psi\in L^{r',p+1}$, so that by dominated convergence theorem  (on the time integral)
$\| \big(F'(\psi)-F'(\psi^n)\big)\partial_t\psi\|_{L^{r',1+1/p}}\to 0$ as $n\to \infty.$
 We conclude that $ \|\psi-\psi^n\|_{W^{1,r}((0,T'),L^{p+1})}\to 0$ as $\|\psi-\psi^n \|_{\DD}\to 0$. An almost identical analysis, starting again from \eqref{pinco}, but this time making use of \ref{propa}-$a)$ and \ref{propa}-$e)$ shows that $ \|\partial_t\psi-\partial_t\psi^n\|_{\infty,2}\to 0$ as $\|\psi_0-\psi_0^n \|_{\DD}\to 0$.
From this last property and again exploiting the equation, we want to deduce finally that  $\|\psi-\psi^n\|_{\DD}\to 0$ as $\|\psi_0-\psi_0^n \|_{\DD}\to 0\ .$ To this end, notice that we have 
\beq\label{contdominio}
\|i\partial_t\psi-i\partial_t\psi^n\|_{\infty,2}=\|\HH\psi+F(\psi)-\HH\psi^n -F(\psi^n)\|_{\infty,2}\to 0
\eeq
Notice that from \eqref{modifiedembeddings} we have $\DD^{s} \hookrightarrow L^{2p} $ with $s$ at least equal to $\frac{n}{4}(1-1/p)$ (and at most $1/2$ or $1/4$ according to dimension $2$ or $3$). From this and  \ref{p:banale}-1)) we obtain continuity of $F:\DD^{s}\to L^2$. On the other hand (see the following Remark \ref{fractional}) from continuity of solution with respect to data in $\X$, already shown, it follows continuity in $\DD^s$, so that from $\|\psi_0^n-\psi_0\|_{\DD^s}\to 0$ we get $\|F(\psi^n)-F(\psi)\|_{\infty,2}\to 0$ and from \eqref{contdominio} we conclude. Finally, to cover the whole interval $[0,T]$ we iterate the argument a finite number of times, possibly extracting a different subsequence from $\{\psi^n_0\}$.
\end{proof}
\begin{remark}\label{fractional}
Notice that from the sole continuity of the solution map in $\X$, by interpolation and use of the uniform bound \eqref{uniform3} we obtain
\[
\|\psi-\psi^n\|_{\DD^s}\leq \|\psi-\psi^n\|^s_{\DD}\|\psi-\psi^n\|^{1-s} \leq C\|\psi_0-\psi^n_0\|^{1-s} \qquad s\in (0,1)
\]
which assures continuity of the solution map with values in $\DD^s$. This does not use the differentiability of the nonlinearity $F$, needed to derive continuity in $\DD$.
\end{remark}
\begin{remark}
Concerning continuity with respect to initial data several results are possible, depending on the functional spaces where continuity is desired. In the previous proof we actually proved that $\|\psi^n-\psi\|_\Z\to 0$ as $\|\psi_0^n-\psi_0\|_{\DD}\to 0\ ,$ which more than stated in Theorem \ref{local}.
\end{remark}

\subsection{Blow-up alternative}
%\begin{prop}{(Blow-up alternative)}
%Let $$T^*=\sup_{T>0}\left\{\psi\in C([0,T],\DD)\cap C^1([0,T], L^2) \ \ \text{solves}\ \ \eqref{eq} \right\}\ .$$
%Then $$\lim_{t\to T^*} \| \psi(t)\|_{\DD}<\infty \ \ \Longrightarrow\  \ T^*=\infty\ .$$
%\end{prop}
\begin{proof}[{\bf Proof of Theorem \ref{local}. Part $5)$}]
 Let us define
 \begin{equation*}%\label{eq:bua}
  M^*:=\sup_{t\in[0,T^*)}\|\Psi(t)\|_{\DD}
 \end{equation*}
 and suppose that $M^*<\infty\ .$\\
 It follows that there exists a sequence of times $\left\{t_n\right\}\subset\RE^+$, $t_n\to T^*$, such that
 \[
  \lim_{n\to \infty}\|\psi(t_n)\|_{\DD}=M\leq M^*.
 \]
 Let $T^*<+\infty$. The local well-posedness proof does not depend on the initial time $t_0$, but only on the fact that $\psi(t_0)\in\DD $. In particular, one sees that from the definition of $M^*$, $\|\psi(t_0)\|_{\DD}\leq M^*$ for every $t_0\in(0,T^*)$. As a consequence the existence time $T(t_0)$ obtained starting from any $t_0\in(0,T^*)$ satisfies
 \[
  T(t_0)\geq C(M^*)>0,
 \]
 for some suitable constant $C(M^*)$ depending only on $M^*$. Setting now $t_0:=t_{{n_0}}$ with $t_{{n_0}}>T^*-C(M^*)$, one concludes that the solution exists beyond $T^*$, which contradicts its definition.
\end{proof}
\begin{remark}
According to the definition of $T^*$, there exists a unique function $\psi\in C([0,T^*);\DD) \cap C^1([0,T^*);L^2)$ coinciding for every $T<T^*$ with the solution $\psi\in C([0,T];\DD) \cap C^1([0,T];L^2)$ of \eqref{eq} as defined by the local existence theorem. The function $\psi$ so defined on $[0,T^*)$ is called the maximal solution of \eqref{eq}.
\end{remark}

\section{Conservation laws and global well posedness \label{s:gwp}}
\subsection{Mass and Energy conservation}
\begin{teo}{(Conservation of Mass and Energy)} In the hypotheses of Theorem \ref{local} we have:
\begin{enumerate}
\item $L^2$- mass is conserved along the evolution: $\| \psi(t) \|^2=\| \psi_0\|^2\qquad \forall t\in [0,T^*) ;$ 
\item Energy is conserved along the evolution: \ \ $E(\psi(t))=E(\psi_0)\qquad  \forall t\in [0,T^*)$\\  where 
\[
E(\psi)=\frac{1}{2}\langle\psi, \HH\psi\rangle \pm \frac{1}{p+1}\|\psi \|^{p+1}_{p+1}\qquad \psi\in \DD.
\]
\end{enumerate}
\end{teo}
\begin{proof}[{\bf Proof}]
Thanks to Corollary \ref{strongeq}, the equation $i\partial_t\psi=\HH\psi \pm|\psi|^{p-1}\psi $ holds as an identity in $L^2$. After taking the inner product with $\psi$ and then the imaginary part of the resulting equation one gets mass conservation.\\
Consider now the energy. Recall that $\HH$ is self-adjoint on $\DD$, the corresponding quadratic form 
 \[%\label{Elin}
 E_{lin}:\DD\rightarrow \RE,\qquad   \psi\mapsto E_{lin}(\psi):=\f{1}{2}\langle\psi,\HH\psi\rangle
 \]
is continuous with respect to the graph norm and differentiable with respect to the $L^2$ norm, with gradient given by $\mathcal H\psi$.\\ Being $p>1$, the same holds true for the nonlinear functional\\  
\[%\label{Enl}
E_{nl}:\DD\rightarrow \RE,\qquad \psi\mapsto E_{nl}(\psi):=\pm\frac{1}{p+1}\|\psi \|^{p+1}_{p+1}
\]
with gradient given by $\pm|\psi|^{p-1}\psi\ .$\\
We can now differentiate with respect to time the total energy along a solution $\psi(t)$ of \eqref{cauchy} and we get
 \begin{align*}
  \frac{d}{dt} E(\psi(t))%= & \Re\{\langle\partial_t\psi(t),\HH\psi(t)\rangle\pm\int |\psi(t)|^{p-1}\overline{\partial_t\psi(t)}\psi(t)\}\\
   =&\Re\{\langle\partial_t\psi(t),\HH\psi(t)\pm |\psi(t)|^{p-1}\psi(t)\rangle\}\\
    =&\Re\{\langle\partial_t\psi(t),i\partial_t\psi(t)\rangle\}=0\ \ \ \ \ \forall t\in(0,T^*)\ .
      %= & \, \re\bigg\{\int_\G\big\langle\partial_t\Psi(t,x),\Dg\Psi(t,x)-|\Psi(t,x)|^{p-2}\Psi(t,x)\big\rangle\dx\bigg\}\\[.2cm]
 % = & \, \re\bigg\{\int_\G\big\langle\partial_t\Psi(t,x),\imath\partial_t\Psi(t,x)\big\rangle\dx\bigg\}=\re\big\{\imath\|\Psi(t)\|_{L^2(\G,\C^2)}^2\big\}=0,
 \end{align*}

\end{proof}

\subsection{Energy bound}
Let us consider the focusing case. Replacing in \eqref{adaptedGN} $q=p+1$ we obtain
\begin{equation}\label{boundLp}
\begin{aligned}
& \|\psi\|_{L^{p+1}}^{p+1} \leq c\|\psi\|^{(1-s)(p+1)}_{L^2} \|\psi\|^{s(p+1)}_{\DD^{1/2}}\qquad  && s\in \Big(\frac{p-1}{p+1},1\Big)\qquad  &d=2; \\
&\|\psi\|_{L^{p+1}}^{p+1} \leq c\|\psi\|^{(1-s)(p+1)}_{L^2} \|\psi\|^{s(p+1)}_{\DD^{1/2}}  && s\in \Big(\frac{3p-3}{2p+2},1/2\Big)  &d=3.
\end{aligned}
\end{equation}
From mass and energy conservation and inequalities \eqref{boundLp} we conclude that both the linear energy $ \langle \psi,\HH \psi\rangle := \|\psi\|^2_{\DD^{1/2}}-\lambda\|\psi\|^2$ and the nonlinear term $\|\psi\|^{p+1}_{L^{p+1}}$, are uniformly bounded in terms of the mass and energy of the initial datum if the quantity $s(p+1)<2$. From the limitation on $s$ this occurs in the $n=2$ case for $p<3$ and in the $n=3$ case for $p<7/3$. These limitations coincide with the ones of the standard NLS equation. Notice however that in the $n=3$ case well posedness in $\DD$ prevents $p\geq 3/2$.
%\diego{Which is the threshold for local well posedness in the form domain when $d=3$? At this point the question is quite natural.}

\subsection{Global existence.} %We borrow ideas from the proof of Th. 5.3.1 in \cite{caz}. 
\begin{teo}{(Global existence)} Let $1<p<3$ if $n=2$ and $1<p<\frac{3}{2}$ if $n=3$ and consider the maximal solution $\psi\in C([0,T^*);\DD) \cap C^1([0,T^*);L^2)$ of the problem \eqref{cauchy}.%Let us denote $p_c=1$ if $n=2$, $p_c=\frac{3}{2}$ if $n=3$. In each of the following cases:\\
%1) $1<p<p_c$;\\
%2) $\|\psi_0\|_{\DD}<\rho$ with $\rho$ sufficiently small;\\
Then the solution $\psi$ is global, i.e. $T^*=\infty$.

\end{teo}
\begin{proof}[{\bf Proof}]
From Cor. \ref{strongeq} we know that the solution $\psi$  is in $ C ( [0,T^*) , \DD) \cap C^1 ( [0,T^*), L^2 (\erre))$, here $T^*$ is the maximal time of existence of the solution. By rephrasing the  blow-up alternative,  we know that if $T^*<\infty$ it must be $\lim_{t\to T^*} \| \psi(t)\|_{\DD}=\infty$. \\
To prove that the solution is global we reason by absurd: we  show that  if $T^* <\infty$ it must hold $\lim_{t\to T^*} \| \psi(t)\|_{\DD}<\infty$; but this contradicts the blow-up alternative and so $T^* = \infty$. \\
Let us assume that $T^*<\infty$. The first observation is that conservation laws imply that 
\begin{equation}\label{apriori}
\|\psi\|_{L^\infty((0,T^*),L^2)} <\infty \qquad \text{and} \qquad  \|\psi\|_{L^\infty((0,T^*),\DD^{1/2})} <\infty . 
\end{equation}
Hence,  to  prove that $\lim_{t\to T^*} \| \psi(t)\|_{\DD}<\infty$ it is enough to show that $\|\HH \psi(t)\|$ does not blow-up in finite time. In particular, we are going to show that 
\[
\|\HH \psi\|_{L^\infty((0,T^*),L^2)}<\infty, 
\]
then, by continuity of $\|\HH \psi(t)\|$, this guarantees that $\lim_{t\to T^*} \| \HH \psi(t)\|<\infty$. \\
Since $\psi$ is a strong solution we have 
\[
\HH \psi(t) = i \partial_t \psi(t) -F(\psi(t)) .%\mp |\psi(t)|^{p-1} \psi(t) \qquad \forall t \in (0,T^*). 
\]
Hence,
\[
\|\HH \psi(t)\| \leq  \|\partial_t \psi(t) \| + \|\psi(t)\|_{2p}^p \qquad \forall t \in (0,T^*),
\]
and, 
\[
\|\HH \psi\|_{L^\infty((0,T^*),L^2)} \leq  \|\partial_t \psi\|_{L^\infty((0,T^*),L^2)} + \|\psi\|_{L^\infty((0,T^*),L^{2p})}^p. 
\]
By the Gagliardo-Nirenberg inequalities in Eq. \eqref{adaptedGN} we infer 
\[ 
 \|\psi(t)\|_{2p}\leq c \|\psi(t)\|_{\DD^{1/2}}  \qquad \forall t \in (0,T^*),
\]
hence, $ \|\psi\|_{L^\infty((0,T^*),L^{2p})} <\infty$  and we  are left to prove that $\partial_t\psi \in L^\infty((0,T^*),L^2)$. Preliminarily we prove that $\partial_t\psi \in L^r((0,T^*),L^{p+1})$. We start by noticing that from part $2)$ of  Th. \ref{local} it follows that
\[ \partial_t \psi \in L^r((0,\tau),L^{p+1}) \qquad \forall 0 < \tau < T^*.\]
Hence, by H\"older inequality in time,
\[ 
\| \partial_t\psi \|_{L^{r'}((0,\tau),L^{p+1})} \leq {T^*}^{1-\frac{2}{r}} \| \partial_t \psi \|_{L^r((0,\tau),L^{p+1})} <\infty ,
\]  
and 
\begin{equation}\label{billy}
\partial_t \psi \in L^{r'}((0,\tau),L^{p+1}) \qquad \forall 0 < \tau < T^*.
\end{equation}
Since, obviously, $\psi(t)$ satisfies also the weak form of the equation, we infer 
\begin{equation}\label{strangle}
\partial_t\psi(t) = -i U(t)\HH \psi_0  - i U(t) F(\psi_0) -i(\Gamma \partial_t F(\psi))(t).
\end{equation}
Next we apply Prop. \ref{propa}-$b)$ and \ref{propa}-$f)$ on $(0,\tau)$ to obtain 
\begin{equation*}% \label{log}
\|\partial_t\psi\|_{L^r((0,\tau),L^{p+1})}  \leq  C(\tau) \big(\|\HH \psi_0\|  + \|F(\psi_0)\| + \|\partial_t F(\psi)\|_{L^{r'}((0,\tau),L^{1+1/p})}\big). 
\end{equation*}
%An  explicit  formula for  $C(\tau)$ is in the proof of Prop. \ref{propa}, where $C(\tau) \leq C \tau^{ \frac1r+\frac1{r'}}$ 
where  $C(\tau)$ is bounded by a constant that depends on $T^*$. Since $\psi_0\in \DD$, also $\|\HH \psi_0\|$ and $\|F(\psi_0)\| $ are bounded and can be absorbed in the constant. Hence, we have the bound
\begin{equation} \label{log2}
\|\partial_t\psi\|_{L^r((0,\tau),L^{p+1})}  \leq  C \big(1+ \|\partial_t F(\psi)\|_{L^{r'}((0,\tau),L^{1+1/p})}\big)
\end{equation}
where the constant $C$ depends on $T^*$ and $\psi_0$ but not on $\tau$. 

By the inequality already used several times before,
\[
\|  F(\psi)(t') - F(\psi)(t) \|_{1+1/p}  \leq C (  \|\psi(t')\|_{p+1}^{p-1} +  \|\psi(t)\|_{p+1}^{p-1})\, \| \psi(t') - \psi(t)\|_{p+1} ,
\]
we obtain
\[
\|  \partial_t  F(\psi)(t) \|_{1+1/p}  \leq C \|\psi(t)\|_{p+1}^{p-1}  \| \partial_t \psi(t)\|_{p+1} \qquad \forall t\in(0,T^*) .
\]
Hence, the bound \eqref{boundLp} and the a-priori bounds \eqref{apriori} give 
\begin{equation}\label{turn}
\|  \partial_t  F(\psi) \|_{L^{r'}((0,\tau),L^{1+1/p})}  \leq C \| \partial_t \psi\|_{L^{r'}((0,\tau),L^{p+1})},
\end{equation}
so that, by inequality \eqref{log2} we infer
\begin{equation}\label{log3}
\|\partial_t\psi\|_{L^r((0,\tau),L^{p+1})}  \leq  C \big(1+\| \partial_t \psi\|_{L^{r'}((0,\tau),L^{p+1})}\big). 
\end{equation}
Fix $0<\ve<\tau<T^*$, and notice that 
\[\begin{aligned}
\| \partial_t \psi\|_{L^{r'}((0,\tau),L^{p+1})} \leq & C \big( \| \partial_t \psi\|_{L^{r'}((0,\tau-\ve),L^{p+1})}+ \| \partial_t \psi\|_{L^{r'}((\tau-\ve,\tau),L^{p+1})} \big) \\
\leq & C \big( \| \partial_t \psi\|_{L^{r'}((0,T^*-\ve),L^{p+1})}+  \ve^{1-\frac2r} \| \partial_t \psi\|_{L^{r}((\tau-\ve,\tau),L^{p+1})} \big)  \\
\leq & C \big(   \| \partial_t \psi\|_{L^{r'}((0,T^*-\ve),L^{p+1})}+ \ve^{1-\frac2r} \| \partial_t \psi\|_{L^{r}((0,\tau),L^{p+1})} \big)  
\end{aligned}\]
Inserting the latter bound in Eq. \eqref{log3}, we obtain the inequality 
\[
\|\partial_t\psi\|_{L^r((0,\tau),L^{p+1})}  \leq  C\big(1+ \| \partial_t \psi\|_{L^{r'}((0,T^*-\ve),L^{p+1})}+ \ve^{1-\frac2r} \| \partial_t \psi\|_{L^{r}((0,\tau),L^{p+1})} \big)   . 
\]
For $\ve$ small enough  the latter term at the r.h.s. can be absorbed in the l.h.s.  to obtain  
\[
\|\partial_t\psi\|_{L^r((0,\tau),L^{p+1})}  \leq  C \big(1+ \| \partial_t \psi\|_{L^{r'}((0,T^*-\ve),L^{p+1})}\big) . 
\]
Since $\partial_t \psi \in {L^{r'}((0,T^*-\ve),L^{p+1})}$ by Eq. \eqref{billy}, and $C$ does not depend on $\tau$, taking the limit  $\tau \to T^*$ we obtain the desired claim $\partial_t\psi \in L^r((0,T^*),L^{p+1})$.

To conclude we go back to Eq. \eqref{strangle} and use Prop. \ref{propa}-$a)$ and $e)$ to obtain 
\[
\|\partial_t\psi\|_{L^\infty((0,T^*),L^2)}  \leq  C \big(\|\HH \psi_0\|  + \|F(\psi_0)\| + \|\partial_t F(\psi)\|_{L^{r'}((0,T^*),L^{1+1/p})}\big) 
\]
As before, see Eq. \eqref{turn},  the bound \eqref{boundLp} and the a-priori bounds \eqref{apriori} give 
\[
\|  \partial_t  F(\psi)\|_{L^{r'}((0,T^*),L^{1+1/p})}  \leq C  \| \partial_t \psi\|_{L^{r'}((0,T^*),L^{p+1})} \leq 
C {T^*}^{1-\frac2r}\| \partial_t \psi\|_{L^{r}((0,T^*),L^{p+1})}, 
\]
hence, $ \partial_t  F(\psi) \in {L^{r'}((0,T^*),L^{1+1/p})}$,  which in turn implies $\partial_t\psi \in {L^\infty((0,T^*),L^2)} $ and concludes the proof.
\end{proof}

\subsection*{Acknowledgments.}
D. Noja acknowledges for funding the EC grant IPaDEGAN (MSCA-RISE-778010). The authors acknowledge the support of the Gruppo Nazionale di Fisica Matematica (GNFM-INdAM).

\end{document}